\documentclass[smallcondensed,final]{svjour3}

\usepackage[colorlinks=true,breaklinks=true,bookmarks=true,urlcolor=blue,
citecolor=blue,linkcolor=blue,bookmarksopen=false,draft=false]{hyperref}

\usepackage{setspace,url}
\hypersetup{colorlinks=true, linkcolor=blue, citecolor=blue, anchorcolor=blue, urlcolor=blue}
\usepackage{comment}
\usepackage{times}

\usepackage{amsmath,amsxtra,amsfonts,amscd,amssymb,bm}
\usepackage{bbm}
\usepackage{booktabs}
\usepackage{supertabular}
\usepackage{multirow}
\usepackage{cases}
\usepackage[usenames]{color}
\usepackage[normalem]{ulem}
\usepackage{graphics,graphicx,epsf,subfigure,epstopdf}
\DeclareMathOperator*{\argmin}{arg\,min}



\newcommand{\zz}{^{\top}}
\newcommand{\nzz}{^{-\top}}
\newcommand{\ct}{^*}

\newcommand{\fs}{^2_{\text{F}}}

\newcommand{\R}{\mathbb{R}}
\newcommand{\bC}{\mathbb{C}}

\newcommand{\inv}{^{-1}}
\newcommand{\nn}{\nonumber}

\newcommand{\mat}[2]{
\left[
\begin{array}{#1}
	#2
\end{array}
\right]
}

\newcommand{\eqnum}[2]{
\begin{eqnarray}\label{#1}
	#2
\end{eqnarray}
}

\newcommand{\eqnonum}[1]{
	\begin{eqnarray*}
		#1
	\end{eqnarray*}
}

\newcommand{\dkh}[1]{\left(#1\right)}

\newtheorem{assumption}{Assumption~}

\newcommand{\red}[1]{\begin{color}{black}#1\end{color}}
\newcommand{\blue}[1]{\begin{color}{black}#1\end{color}}
\newcommand{\magenta}[1]{\begin{color}{black}#1\end{color}}
\definecolor{mygreen}{rgb}{0,.5,0}

\newcommand{\hQ}{\hat{Q}}
\newcommand{\hA}{\hat{A}}
\newcommand{\hB}{\hat{B}}
\newcommand{\hl}{\hat{\lambda}}
\newcommand{\eig}[1]{\mathrm{eig}(#1)}

\usepackage{cite}

\begin{document}

\authorrunning{C. Chen, M. Li, X. Liu and Y. Ye}
\titlerunning{Extended ADMM and BCD for Nonseparable Convex Minimization Models}
\title{Extended ADMM and BCD for Nonseparable Convex Minimization Models with Quadratic Coupling Terms: Convergence Analysis and Insights}

\author{
    Caihua Chen  \and
    Min Li \and
    Xin Liu \and
    Yinyu Ye
}

\institute{
	C. Chen \at International Center of Management Science and Engineering,
	School of Management and Engineering, Nanjing University, China, \email{chchen@nju.edu.cn}
	\and
	M. Li \at School of Economics and Management, Southeast University, China,
	\email{limin@seu.edu.cn}
	\and
	X. Liu \at State Key Laboratory of Scientific and Engineering Computing, Academy of Mathematics and Systems Science, Chinese Academy of Sciences, and University of Chinese Academy of Sciences, China, \email{liuxin@lsec.cc.ac.cn}
	\and
	Y. Ye \at Department of Management Science and Engineering, School of Engineering,
	Stanford University, USA; and International Center of Management Science and Engineering, School of Management and Engineering, Nanjing University, \email{yyye@stanford.edu}
}

\date{Received: date / Accepted: date}
\maketitle

\begin{abstract}
In this paper,
we establish the convergence of the \red{proximal} alternating direction method of multipliers (ADMM) and block coordinate descent (BCD) for %{large-scale}
nonseparable minimization models with quadratic coupling terms. The novel convergence results presented in this paper answer several open questions that have been the subject of considerable discussion. We firstly extend the 2-block \red{proximal} ADMM to linearly constrained convex optimization with a coupled quadratic objective function, an area where theoretical understanding is currently lacking, and prove that the sequence generated by the \red{proximal} ADMM converges in point-wise manner to a primal-dual solution pair. Moreover, we apply randomly permuted ADMM (RPADMM) to nonseparable multi-block convex optimization, and prove its expected convergence for a class of nonseparable quadratic programming problems. When the linear constraint vanishes, the 2-block \red{proximal} ADMM and RPADMM reduce to the 2-block cyclic \red{proximal} BCD method and randomly permuted BCD (RPBCD). Our study provides the first iterate convergence result for 2-block cyclic proximal BCD without assuming the boundedness of the iterates. %Under the same setting, the sublinear convergence rate of the function values can also be verified.
We also theoretically establish the expected iterate convergence result concerning multi-block RPBCD for convex quadratic optimization. In addition, we demonstrate that RPBCD may have a worse convergence rate than cyclic proximal BCD for 2-block convex quadratic minimization problems. Although the results on RPADMM and RPBCD are restricted to quadratic minimization models, they provide some interesting insights: 1) random permutation makes ADMM and BCD more robust for multi-block convex minimization problems; 2) cyclic BCD may outperform RPBCD for ``nice'' problems, and therefore RPBCD should be applied with caution when solving general convex optimization problems.
\end{abstract}

\medskip

\keywords{Nonseparable convex minimization, Alternating direction method of
    multipliers, Block coordinate descent method, Iterate convergence,
    Random permutation}

\subclass{65K05, 90C26}

\section{Introduction}

In this paper we consider the linearly constrained convex minimization model with an objective function that is the sum of several separable functions and a coupled quadratic function:
\begin{eqnarray}\label{eq:convobj}\label{eq:obj1}
\begin{array}{ll}
\min\limits_{x \in \R^d} &  \displaystyle \theta(x): = \sum_{i=1}^n \theta_i(x_i) + \frac{1}{2}x^\top H x + g^\top x\\ [0.2cm]
\mbox{s.t.} &  {\displaystyle  \sum_{i=1}^n A_ix_i  =b,}
\end{array}
\end{eqnarray}
where $\theta_i:\R^{d_i}\mapsto (-\infty,\,+\infty]$ ($i=1,2,\ldots,n$)
are closed
proper convex (not necessarily smooth) functions; $x_i \in \R^{d_i},
x=(x_1,x_2,\ldots,x_n)\in \R^d$; $H\in \R^{d\times d}$ is a
symmetric \red{and} positive semidefinite matrix; $g\in \R^d$; $A_i\in
\R^{m\times d_i}$ and $b\in \R^m$. A point $(\bar x, \bar \mu)$ is
said to be a Karush-Kuhn-Tucker (KKT) point of \eqref{eq:obj1} if it satisfies
\begin{equation}\label{kkt}
\left\{
\begin{array}{rcl}
- (H \bar x+g)_i +  A_i^\top \bar \mu  &\in& \partial \theta_i(\bar x_i), \quad i=1,\cdots,n,\\[0.2cm]
\sum_{i=1}^n A_i \bar{x}_i  &=&  b.
\end{array}
\right.
\end{equation}
The set consisting of the KKT points of \eqref{eq:obj1} is assumed
to be nonempty. Problem \eqref{eq:obj1} has many applications in signal and imaging processing, machine learning, statistics, and engineering; e.g., see
\cite{MaYi,CuiLiSunToh2015,agarwal2012noisy,Hong2014,Feng14,Mota}.
%Once $H\equiv 0$ and $g\equiv 0$, namely, the coupled objective vanishes and
%\eqref{eq:obj1} reduces to a separable objective.

The augmented Lagrangian function of \eqref{eq:convobj} is
\begin{eqnarray}\label{eq:lag}
\mathcal{L}_\beta (x_1, \ldots, x_n;\mu) := \displaystyle
\sum_{i=1}^n \theta_i(x_i) + \frac{1}{2}x^\top H x + g^\top x - \mu\zz
\big(\sum_{i=1}^n A_ix_i-b\big) + \frac{\beta}{2}\big\|\sum_{i=1}^n
A_ix_i - b\big\|^2,
\end{eqnarray}
where $\mu\in\R^m$ is the Lagrangian multiplier and $\beta>0$ is the
penalty parameter. In this paper, we extend the $n$-block
\red{proximal} alternating direction method of multipliers (ADMM) to solve the nonseparable
convex minimization problem \eqref{eq:obj1}, which
consists of a cyclic update of the primal variables $x_i\,
(i=1,2,\ldots,n)$ in the Gauss-Seidel fashion and a dual ascent type
update of $\mu$ at each iteration, i.e.,
\red{\begin{eqnarray}\label{eq:admm}
\left\{
\begin{array}{l}
\displaystyle x_1^{k+1} := \argmin\limits_{x_1 \in \R^{d_1}} \Big\{\mathcal{L}_\beta (x_1, x_2^k, \ldots, x_n^k;\mu^k) + \frac{1}{2}\|x_1 - x_1^k\|^2_{R_1}\Big\},\\
\displaystyle x_2^{k+1} :=\argmin\limits_{x_2  \in \R^{d_2}} \Big\{\mathcal{L}_\beta (x_1^{k+1}, x_2, x_3^k, \ldots,x_n^k;\mu^k) + \frac{1}{2}\|x_2 - x_2^k\|^2_{R_2}\Big\},\\
\displaystyle   \cdots\cdots  \\
\displaystyle x_n^{k+1} := \argmin\limits_{x_n  \in \R^{d_n}} \Big\{\mathcal{L}_\beta (x_1^{k+1}, x_2^{k+1}, \ldots, x_{n-1}^{k+1}, x_n;\mu^k)+ \frac{1}{2}\|x_n - x_n^k\|^2_{R_n}\Big\},\\
\displaystyle \mu^{k+1} := \mu^k -\beta (\sum_{i=1}^n
A_ix_i^{k+1}-b),
\end{array}
\right.
\end{eqnarray}
where $R_i \in \R^{d_i \times d_i}$, $i=1, \cdots, n$, are
symmetric and positive semidefinite matrices.}
%   It is not difficult to verify that the well-definedness of \eqref{eq:admm}
%   can be guaranteed by the existence of KKT point of \eqref{eq:obj1}, which is a basic assumption
%   throughout this paper.

Note that the algorithmic scheme \eqref{eq:admm}  reduces to the classical ADMM when there are only \red{two blocks ($n=2$), the coupled objective vanishes ($H=0$ and $g=0$) and $R_i = 0$ ($i=1, 2$).} ADMM was originally introduced in the early 1970s  \cite{Glow1975,Gabay1976}, and its convergence propertites have been studied extensively in the literature \cite{Glow1984,Eckstein1992,Bertsekas1997,He2012,MSvaiter2013,DYin,Deng2012}. Because of its wide versatility and applicability in multiple fields, ADMM is a popular means of solving optimization problems, especially those related to big data; we refer to\cite{Boyd2011} for a survey on the modern applications of ADMM.

For the case of $n\geq 3$, numerous research efforts have been devoted to analyzing the convergence of multi-block ADMM and its variants for the linearly constrained separable convex optimization model, i.e.,  \eqref{eq:convobj} without the coupled term.
Recent work \cite{CHYY2013} has shown that the $n$-block ADMM
\eqref{eq:admm} is not necessarily convergent, even for a nonsingular square system of linear equations. Various methods have been proposed to overcome the divergence issue of multi-block ADMM. One typical solution is to combine correction steps with the output of  $n$-block ADMM
 \eqref{eq:admm} \cite{hty2015,he2012alternating,HYZC}. If at least $n-2$ functions in the objective are strongly convex, it has been shown that  \eqref{eq:admm} is globally convergent, provided that the penalty parameter $\beta$  is restricted to a specific range \cite{Han2012,chen2013convergence,cai2014direct,lin2014convergence,Zhang2010,li2015convergent}. Without strong convexity, it has been shown \cite{hong2012linear} that the $n$-block ADMM with a small dual stepsize,
where the multiplier update \eqref{eq:admm} is replaced by
\[\mu^{k+1}=\mu^k - \tau \beta(\sum_{i=1}^n A_ix_i^{k+1} - b),
\]
is linearly convergent provided that the objective function satisfies certain error bound conditions.
{Some very recent studies \cite{lin15multiblock2,lin14linear}
have demonstrated the convergence of multi-block ADMM under some other conditions, and some convergent proximal variants of the multi-block ADMM have been proposed for solving convex linear/quadratic conic programming problems  \cite{sun2015convergent,LiST2014convergent,ChenSunToh2015-1}. A recent paper \cite{SLY2015} proposed a randomly modified variant of the multi-block ADMM \eqref{eq:admm},
called randomly permuted ADMM (RPADMM). At each step, RPADMM forms a random permutation of  $\{1,2,\ldots,n\}$ (known as block sampling without replacement), and updates the primal variables  $x_i\,(i=1,2,\ldots,n)$
in the order of the chosen permutation followed by the regular multiplier update. Surprisingly, RPADMM is convergent in expectation for any nonsingular square system of linear equations
\cite{SLY2015}.

In contrast to the separable case, studies on the convergence properties of $n$-block ADMM for
\eqref{eq:convobj} with nonseparable objective, even for $n=2$, are limited. In
\cite{Hong2014},  the authors demonstrated that when problem
\eqref{eq:convobj} is convex but not necessarily separable\footnote{
 The models considered in \cite{Hong2014,HLR2014}  are more general than problem
\eqref{eq:convobj}, as the authors of  \cite{Hong2014,HLR2014} actually allow
generally nonseparable smooth function in the objective, but in  \eqref{eq:convobj}
the coupled objective is a quadratic function.}, and certain error bound conditions are satisfied, the ADMM iteration converges to some primal-dual optimal solution, provided that the stepsize in the update of the multiplier is sufficiently small. Despite this conservative nature, the stepsize usually depends on some unknown parameters associated with the error bound, and may thus be difficult to compute, which often makes the algorithm less efficient. In view of this, it might be more beneficial to employ the classical ADMM \eqref{eq:admm} (with $\tau =1$) or its variants with a large stepsize $\tau\geq 1$. However, as mentioned in \cite{HLR2014}, {\bf``when the objective function is not separable across the variables, %(e.g., the coupling function $l(x)$ appears in the objective),
the convergence of the ADMM  \eqref{eq:admm} is still open, even in
the case where $n=2$ and $\theta(\cdot)$ is convex."} Along slightly different lines, \cite{CuiLiSunToh2015}  investigated the convergence of a majorized ADMM for the convex optimization problem with a coupled smooth objective function, which includes the $2$-block  ADMM
\eqref{eq:admm} for \eqref{eq:obj1} as a special case. Convergence was established for the case when the subproblems of the ADMM admit unique solutions and $H, A_1, A_2, R_1$  and $R_2$ satisfy some additional
restrictions; see Remark 4.2 in  \cite{CuiLiSunToh2015} for details.
Very recently, \cite{GZhang2015} studied the convergence and
ergodic complexity of a $2$-block proximal ADMM and its variants for
the nonseparable convex optimization by assuming some additional
conditions on the problem data. As the positive definite proximal
terms are indispensable in the analysis of these algorithms, the results derived in \cite{GZhang2015} are not  applicable to the
scheme \eqref{eq:admm} for problem \eqref{eq:convobj} since $R_1$ and $R_2$ are only positive semidefinite.

In this paper, we analyze the iterate convergence of proximal ADMM  \eqref{eq:admm} and the  randomly permuted ADMM for solving
the nonseparable convex optimization problem \eqref{eq:convobj}. The
main contributions of our paper are threefold. Firstly, we prove that
the 2-block proximal ADMM is convergent for \eqref{eq:convobj} only under a condition that ensures the subproblems have unique solutions. Our condition is the weakest to ensure iterate convergence for the proximal ADMM since, as we will see in  Section 2, it is not only
sufficient but also necessary for the convergence of the proximal ADMM
applied to some special problems. Our analysis partially answers the
open question mentioned in \cite{HLR2014} on the convergence of ADMM
for nonseparable convex optimization problems. Secondly, we extend the
RPADMM  proposed in \cite{SLY2015} to solve the model
\eqref{eq:convobj}, and prove its expected convergence in the case
where $\theta_i \equiv 0\,(i=1,2,\ldots,n)$. This
result is a non-trivial extension of the convergence result shown in
\cite{SLY2015}, since the objective in \eqref{eq:convobj} is more general and
its solution set may not be a singleton. Thirdly, when restricted to the unconstrained case, that is,  $A_i$ ($i=1,\cdots,n$) and $b$ are absent,  the proximal ADMM and RPADMM reduce to the
cyclic proximal block coordinate descent (BCD) method (also known as the alternating minimization method), i.e.,
\red{\begin{eqnarray}\label{eq:BCD}
\left\{
\begin{array}{l}
\displaystyle x_1^{k+1} := \argmin\limits_{x_1 \in \R^{d_1}} \theta (x_1, x_2^k, \ldots, x_n^k)+ \frac{1}{2}\|x_1 - x_1^k\|^2_{R_1},\\
\displaystyle x_2^{k+1} :=\argmin\limits_{x_2  \in \R^{d_2}}  \theta (x_1^{k+1}, x_2, x_3^k, \ldots,x_n^k)+ \frac{1}{2}\|x_2 - x_2^k\|^2_{R_2},\\
\displaystyle   \cdots\cdots  \\
\displaystyle x_n^{k+1} := \argmin\limits_{x_n  \in \R^{d_n}} \theta(x_1^{k+1}, x_2^{k+1}, \ldots, x_{n-1}^{k+1}, x_n)+ \frac{1}{2}\|x_n - x_n^k\|^2_{R_n}\Big\}.
\end{array}
\right.
\end{eqnarray}}
and randomly permuted BCD. An implication of our work is the iterate
convergence of the $2$-block cyclic proximal BCD method for the whole sequence and,
in particular, the expected convergence of randomly permuted multi-block BCD. Although the literature on BCD-type methods is vast  (e.g., \cite{Bert1999,Tseng2001,TYun2009,Beck2013,Luxiao2013,RHLuo2013,shalev2013stochastic,Beck2015,shefi15}),
there are very few results on the iterate convergence of BCD-type methods. As mentioned  in \cite{BST2013}, \textbf{``in all these works [on BCD or its proximal variants] only convergence of the subsequences can be established.''}
By assuming that the Kurdyka-\L{}ojasiewicz property holds on the objective function and the iterates are bounded,  \cite{AB10} and \cite{BST2013} established the iterate convergence of the proximal BCD and proximal alternating linearized minimization, respectively. It is clear that these results are also applicable to the BCD type methods for convex minimization problems. While the boundedness assumption of the sequence are typical to establish the iterate convergence of algorithms for nonconvex optimization problems, it might be a bit restrictive to assume the boundedness for analyzing the iterate convergence for the convex cases.
To the best of our knowledge, our convergence result for the  $2$-block \red{proximal} BCD method is the first for the %original 
\red{proximal} BCD that only requires the unique solutions-type condition of the subproblems, rather than any assumptions on the boundedness of the iterates. %Moreover, the sublinear convergence rate of the function values can also be obtained without the assumption on the sequence boundedness.

It has been claimed that randomly permuted BCD (RPBCD, also known as the ``sampling without replacement'' variant of randomized BCD, and called ``EPOCHS'' in a recent survey  \cite{Steve2015}) tends to converge faster than the randomized BCD \cite{Steve2015} , with the classical cyclic version performing even worse. Some numerical advantages of RPBCD compared with randomized BCD and cyclic BCD were discussed in \cite{shalev2013stochastic}. In fact, it has been stated that ``this kind of randomization [RPBCD] has been shown in several contexts to be superior to the sampling with replacement scheme analyzed above, but a theoretical understanding of this phenomenon remains elusive'' \cite{Steve2015}.   Randomized BCD (``sampling with replacemen'') has already been extensively studied \cite{Peter2014}, but its theoretical analysis does not apply to RPBCD. Although the function value convergence results  \cite{Tseng2001,Beck2013,HongBCD}
for cyclic or essential cyclic BCD can be simply extended to RPBCD, these analysis techniques are independent of permutation, so there remains a lack of direct theoretical analysis on the iterate convergence of RPBCD. Our expected iterate convergence of RPBCD for quadratic minimization problems can be regarded as the first direct analysis on the iterate convergence of the ``sampling without replacement'' variant of randomized BCD. We also prove that RPBCD may have a worse convergence rate than cyclic BCD for quadratic minimization problems. Thus, RPBCD should be used with caution for solving general optimization problems.

The rest of this paper is organized as follows. In Section 2, we
prove the iterate convergence of the $2$-block \red{proximal} ADMM and  cyclic
BCD for linearly constrained optimization problems with a coupled quadratic objective function \eqref{eq:convobj} and its unconstrained
variant, respectively. Section 3 illustrates the expected
convergence of the RPADMM and the RPBCD for a class of linear constrained quadratic optimization problems and its unconstrained variant, respectively. Finally, we conclude our paper and present some insights into the use of ADMM and BCD in  Section 4.

\red{\section{Convergence of $2$-Block Proximal ADMM}}

In this section, we will specify $n=2$ and analyze the iterate convergence of the 2-block \red{proximal} ADMM for
 the convex optimization model \eqref{eq:convobj}. For notational simplicity, we write
 \[  H: = \left[
\begin{array}{cc}
H_{11}       &     H_{12}   \\
H_{12}^\top  &     H_{22}
\end{array}
\right], \qquad  R: = \left[
\begin{array}{cc}
R_1          &     0   \\
0            &     R_2
\end{array}
\right] \qquad \hbox{and}  \qquad g: = \left[
\begin{array}{c}
g_1  \\
g_2
\end{array}
\right],
\nn\]
 and define the quadratic function $\phi(x_1,x_2)$ by
\begin{equation}\label{def-phi} \phi(x_1,x_2): = \frac{1}{2}x_1^\top H_{11}x_1 + x_1^\top H_{12} x_2 + \frac{1}{2}x_2^\top H_{22} x_2
+ g_1^\top x_1 + g_2^\top x_2.
\end{equation}
Thus the problem under consideration can be written as
\begin{eqnarray}\label{min-problem}
\begin{array}{ll}
\displaystyle\min_{x\in\R^d} &
\displaystyle \theta(x):=  \theta_1(x_1) + \theta_2(x_2) +  \phi(x_1,x_2)  \\[5pt]
\mbox{s.t.}
&   A_1 x_1 + A_2 x_2 =b.
\end{array}
\end{eqnarray}

Since $\theta_1$ and $\theta_2$ are closed convex functions,
there exist two symmetric positive semidefinite matrices $\Sigma_1$
and $\Sigma_2$ such that
\begin{equation}\label{def-sig-p} (x_1 - \hat{x}_1)^\top (w_1 -
\hat{w}_1) \ge \|x_1 - \hat{x}_1\|^2_{\Sigma_1},
\quad \forall\; x_1, \hat{x}_1 \in  {\rm dom}(\theta_1), \; w_1 \in \partial \theta_1(x_1), \hat{w}_1 \in \partial \theta_1(\hat{x}_1)
\end{equation}
and
\begin{equation}\label{def-sig-q}         (x_2 -  \hat{x}_2)^\top (w_2 - \hat{w}_2) \ge \|x_2 -  \hat{x}_2\|^2_{\Sigma_2},
\quad \forall\; x_2, \hat{x}_2 \in  {\rm dom}(\theta_2), \; w_2 \in \partial \theta_2(x_2), \hat{w}_2 \in \partial \theta_2(\hat{x}_2),
\end{equation}
where $\partial \theta_1$ and $\partial \theta_2$ are the subdifferential
mappings  of $\theta_1$ and $\theta_2$, respectively. By letting
\begin{equation}\label{def-sigma} x : = \left[\begin{array}{c}
x_1  \\
x_2
\end{array}\right], \;\;
\hat{x}:=
\left[\begin{array}{c}
\hat{x}_1  \\
\hat{x}_2
\end{array}\right],
\;\;
w : = \left[\begin{array}{c}
w_1  \\
w_2
\end{array}\right],
\;\;
\hat{w} : = \left[\begin{array}{c}
\hat{w}_1  \\
\hat{w}_2
\end{array}\right]
\,\, \hbox{and} \,\,
\Sigma: = \left[\begin{array}{cc}
\Sigma_1 &  0  \\
0       &  \Sigma_2
\end{array}\right],
\end{equation} we have
\begin{equation}\label{convexity}
(x -  \hat{x})^\top (w - \hat{w})
\ge \|x  - \hat{x} \|^2_{\Sigma}.
\end{equation}

The following lemma establishes the contraction property with respect to the solution set of
\eqref{min-problem} for the sequence generated by \eqref{eq:admm}, which plays an
important role in the subsequent analysis.

\begin{lemma}\label{Contraction}\label{lemma:contra}
Assume the $2$-block \red{proximal} ADMM  \eqref{eq:admm} is well defined for problem \eqref{min-problem}. Let $\{(x_1^k, x_2^k, \mu^k)\}$ be the sequence
generated by  \eqref{eq:admm}. Then, the following statements hold.
\item[{\rm (i)}] If  $(\bar x_1, \bar x_2,\bar \mu)$ is any given KKT point of problem \eqref{min-problem}, then
we have
\red{\begin{eqnarray}\label{contraction-ine}
&  & \Big(\frac{7}{8}\|x^k - \bar x\|_{H + \Sigma + \frac{4}{7}R}^2 +
\frac{1}{2}\|x_2^k -
\bar x_2\|_{H_{22}+\Sigma_2+\beta A_2^\top A_2}^2 + \frac{1}{2\beta}\|\mu^k - \bar \mu\|^2 + \frac{1}{2}\|x_2^k - x_2^{k-1}\|^2_{R_2}\Big) \nonumber \\
&  & \; - \Big(\frac{7}{8}\|x^{k+1} - \bar x\|_{H + \Sigma + \frac{4}{7}R}^2 +
\frac{1}{2}\|x_2^{k+1} -
\bar x_2\|_{H_{22}+\Sigma_2+\beta A_2^\top A_2}^2 + \frac{1}{2\beta}\|\mu^{k+1} - \bar \mu\|^2  + \frac{1}{2}\|x_2^{k+1} - x_2^k\|^2_{R_2}\Big) \nonumber \\
& & \ge \frac{1}{16}\|x^{k+1} - x^k\|^2_{H+\Sigma+8R} +
\frac{1}{6}\|x_2^{k+1} - x_2^k\|^2_{H_{22} + \Sigma_2 + 3\beta
A_2^\top A_2}
    + \frac{1}{2\beta}\|\mu^{k+1} - \mu^k\|^2.
\end{eqnarray}}
\item [{\rm (ii)}] It holds that
\begin{equation}\label{KKT}
\left\{
\begin{array}{l}
\displaystyle \lim_{k \rightarrow \infty} d(0, \; \partial \theta_1(x_1^{k+1})+ \nabla_{x_1}\phi(x_1^{k+1}, x_2^{k+1})- A_1^\top\mu^{k+1})= 0, \\[0.2cm]
\displaystyle \lim_{k \rightarrow \infty} d(0, \; \partial \theta_2(x_2^{k+1})+ \nabla_{x_2}\phi(x_1^{k+1}, x_2^{k+1})- A_2^\top\mu^{k+1})= 0, \\[0.2cm]
\displaystyle \lim_{k \rightarrow \infty}\|A_1x_1^{k+1} + A_2 x_2^{k+1} -b\| = 0,
\end{array}
\right.
\end{equation}
where $d(\cdot, \cdot)$ denotes the Euclidean distance of some point
to a set.
\end{lemma}

\begin{proof}(i) \red{From the first order optimality condition of \eqref{eq:admm}, we get}
\red{\begin{equation*}
\left\{
\begin{array}{l}
0 \in \partial \theta_1(x_1^{k+1}) + \nabla_{x_1}\phi(x_1^{k+1}, x_2^k) - A_1^\top \mu^k + \beta A_1^\top(A_1x_1^{k+1} + A_2x_2^k - b) + R_1(x_1^{k+1} - x_1^k), \\[0.2cm]
0 \in \partial \theta_2(x_2^{k+1}) + \nabla_{x_2}\phi(x_1^{k+1}, x_2^{k+1})
- A_2^\top \mu^k + \beta A_2^\top(A_1x_1^{k+1} + A_2x_2^{k+1} - b)+ R_2(x_2^{k+1} - x_2^k),
\end{array}
\right.
\end{equation*}}
where $\phi(\cdot, \cdot)$ is defined in \eqref{def-phi}. Using
the definitions of $\phi$ and $\mu^{k+1}$, the above formulas imply that
\red{\begin{equation}\label{opt-cond}
\left\{
\begin{array}{l}
- \nabla_{x_1}\phi(x_1^{k+1}, x_2^{k+1}) +   A_1^\top \mu^{k+1} + (H_{12} + \beta A_1^\top A_2)(x_2^{k+1} - x_2^k) - R_1(x_1^{k+1} - x_1^k) \in \partial \theta_1(x_1^{k+1}), \\[0.2cm]
- \nabla_{x_2}\phi(x_1^{k+1}, x_2^{k+1}) +   A_2^\top \mu^{k+1} - R_2(x_2^{k+1} - x_2^k) \in
\partial \theta_2(x_2^{k+1}).
\end{array}
\right.
\end{equation}}
Since $(\bar x_1, \bar x_2,\bar \mu)$ is a KKT point of  \eqref{min-problem}, we have that
\begin{equation}\label{ine-pq}
\left\{
\begin{array}{l}
- \nabla_{x_1}\phi(\bar x_1, \bar x_2) +   A_1^\top \bar \mu  \in \partial \theta_1(\bar x_1), \\[0.2cm]
- \nabla_{x_2}\phi(\bar x_1, \bar x_2) +   A_2^\top \bar \mu  \in \partial \theta_2(\bar x_2), \\[0.2cm]
A_1\bar x_1 + A_2\bar x_2 =b.
\end{array}
\right.
\end{equation}
From \eqref{convexity}, \eqref{opt-cond} and \eqref{ine-pq}, we
obtain
\red{\begin{eqnarray}\label{ppqq-ine}
&   &  \hspace{-0.9cm}  \|x^{k+1} - \bar x\|^2_{\Sigma} \nn \\
&  &  \hspace{-0.9cm} \le (x_1^{k+1} - \bar x_1)^\top \Big\{\big[- \nabla_{x_1}\phi(x_1^{k+1}, x_2^{k+1}) +   A_1^\top \mu^{k+1} + (H_{12} + \beta A_1^\top A_2)(x_2^{k+1} - x_2^k)- R_1(x_1^{k+1} - x_1^k) \big] \nn \\
&  &  \hspace{-0.4cm} - \big[- \nabla_{x_1}\phi(\bar x_1, \bar x_2) +   A_1^\top
\bar \mu\big]\Big\} +   (x_2^{k+1} - \bar x_2)^\top \Big\{\big[-
\nabla_{x_2}\phi(x_1^{k+1}, x_2^{k+1}) +
A_2^\top \mu^{k+1}- R_2(x_2^{k+1} - x_2^k)  \big]  \nn \\
&  &  \hspace{-0.4cm}  -\big[- \nabla_{x_2}\phi(\bar x_1, \bar x_2) +   A_2^\top \bar \mu \big]   \Big\} \nonumber \\
 &   &  \hspace{-0.9cm} = - (x_1^{k+1} - \bar x_1)^\top A_1^\top (\bar \mu - \mu^{k+1})- (x_2^{k+1} - \bar x_2)^\top A_2^\top (\bar \mu - \mu^{k+1})- (x^{k+1} - \bar x)^TR(x^{k+1} -x^k) \nonumber \\
 &     &  \hspace{-0.4cm}    + (x_1^{k+1} - \bar x_1)^\top (H_{12} + \beta A_1^\top A_2)(x_2^{k+1} - x_2^k) -(x^{k+1} - \bar x)^\top \big(\nabla \phi(x_1^{k+1}, x_2^{k+1})- \nabla \phi(\bar x_1, \bar x_2)\big) \nonumber \\
 &   & \hspace{-0.9cm} = (x^{k+1} -x^k)^TR(\bar x - x^{k+1}) + \frac{1}{\beta}(\mu^{k+1} - \mu^k)^\top (\bar \mu - \mu^{k+1})  + (x_1^{k+1} - \bar x_1)^\top (H_{12} + \beta A_1^\top A_2) (x_2^{k+1} - x_2^k) \nonumber \\
 &     &  \hspace{-0.4cm}
 - \|x^{k+1} - \bar x\|^2_H.
\end{eqnarray}}
By simple manipulations and using $A_1\bar x_1 + A_2 \bar x_2 =
b$, we can see that
\begin{eqnarray}\label{x1x2-ine}
&  & \beta (x_1^{k+1} - \bar x_1)^\top A_1^\top A_2(x_2^{k+1} - x_2^k) \nonumber \\
& = & -\beta(A_2x_2^{k+1} - A_2 \bar{x}_2)^\top(A_2x_2^{k+1} -
A_2 x_2^k) +
\beta(A_1 x_1^{k+1} + A_2 x_2^{k+1} - b)^\top (A_2 x_2^{k+1} - A_2x_2^k)\nonumber\\
& = &  \frac{\beta}{2}(\|A_2x_2^k - A_2\bar x_2\|^2 - \|A_2x_2^{k+1}
- A_2\bar x_2\|^2) - \frac{\beta}{2}\|A_2x_2^{k+1}
       - A_2x_2^k\|^2 \nonumber \\
&  & \quad + \beta(A_1 x_1^{k+1} + A_2x_2^{k+1} - b)^\top
(A_2x_2^{k+1} - A_2 x_2^k),
\end{eqnarray}
\red{\begin{equation} \label{x-equ}  (x^{k+1} -x^k)^TR(\bar x - x^{k+1})
= \frac{1}{2}(\|x^k - \bar x\|_R^2 - \|x^{k+1} - \bar x\|_R^2 - \|x^{k+1} - x^k\|_R^2)
\end{equation}}
and
\begin{equation} \label{mu-ine} \frac{1}{\beta}(\mu^{k+1} - \mu^k)^\top (\bar \mu - \mu^{k+1})
   = \frac{1}{2\beta}(\|\mu^k - \bar \mu\|^2 - \|\mu^{k+1} - \bar \mu\|^2 - \|\mu^{k+1} - \mu^k\|^2). \end{equation}
  On the other hand, it follows from \eqref{opt-cond} that
\red{\begin{equation*}\left\{
	\begin{array}{l}
- \nabla_{x_2}\phi(x_1^{k+1}, x_2^{k+1}) + A_2^\top \mu^{k+1} - R_2(x_2^{k+1} - x_2^k)\in
\partial \theta_2(x_2^{k+1}),  \\[0.2cm]
-\nabla_{x_2}\phi(x_1^k, x_2^k) + A_2^\top \mu^k - R_2(x_2^k - x_2^{k-1})\in \partial
\theta_2(x_2^k),
\end{array}
\right.
\end{equation*}}
which, together with \eqref{def-sig-q}, implies
\red{\begin{eqnarray}\label{ine-phi}
&  &  (x_2^{k+1}- x_2^k)^\top \big[ - \nabla_{x_2}\phi(x_1^{k+1},
x_2^{k+1}) + A_2^\top \mu^{k+1} - R_2 (x_2^{k+1} - x_2^k) +  \nabla_{x_2}\phi(x_1^k, x_2^k) -
A_2^\top
\mu^k + R_2 (x_2^k - x_2^{k-1}) \big]  \nonumber \\
&  & \quad \ge \|x_2^{k+1} - x_2^k\|^2_{\Sigma_2}.
\end{eqnarray}}
\red{Recall that
\[ \mu^{k+1} - \mu^k = -\beta(A_1 x_1^{k+1} + A_2 x_2^{k+1} - b)  \qquad \hbox{and} \qquad
 \nabla_{x_2}\phi(x_1, x_2) = H_{12}^\top x_1 + H_{22}x_2 + g_2. \]}
\red{Then, by using Cauchy-Schwarz inequality, the inequality
\eqref{ine-phi}} gives
\red{\begin{eqnarray}
&  &  \beta(A_1 x_1^{k+1} + A_2 x_2^{k+1} - b)^\top (A_2x_2^{k+1} -
A_2x_2^k) \nonumber\\
&  & \; \le - \|x_2^{k+1} - x_2^k\|_{H_{22}+\Sigma_2}^2 +
(x_2^{k+1}- x_2^k)^\top H_{12}^\top (x_1^k- x_1^{k+1})
- \|x_2^{k+1} - x_2^k\|^2_{R_2} + (x_2^{k+1} - x_2^k)^TR_2(x_2^k - x_2^{k-1}) \nonumber\\
&  & \; \le - \|x_2^{k+1} - x_2^k\|_{H_{22}+\Sigma_2}^2 +
(x_2^{k+1}- x_2^k)^\top H_{12}^\top (x_1^k- x_1^{k+1})
- \frac{1}{2}\|x_2^{k+1} - x_2^k\|^2_{R_2} +\frac{1}{2} \|x_2^k - x_2^{k-1}\|_{R_2}^2. \nonumber
\end{eqnarray}}
Substituting \eqref{x1x2-ine}, \red{\eqref{x-equ},} \eqref{mu-ine} and the above
inequality into \eqref{ppqq-ine}, we further get
\red{\begin{eqnarray}\label{contra-ine1}
&  & \frac{1}{2}\big(\|x^k - \bar x\|^2_R - \|x^{k+1} -
\bar x\|^2_R\big) + \frac{1}{2\beta}\big(\|\mu^k - \bar \mu\|^2 - \|\mu^{k+1} -
\bar \mu\|^2\big)
 + \frac{\beta}{2}\big(\|A_2 x_2^k - A_2 \bar x_2\|^2 \nonumber \\
&  & \quad - \|A_2 x_2^{k+1} - A_2 \bar x_2 \|^2\big) + \frac{1}{2}\big(\|x_2^k - x_2^{k-1}\|_{R_2}^2 -  \|x_2^{k+1} - x_2^k\|^2_{R_2}\big) \nonumber \\
&  & \ge \|x^{k+1} - \bar x\|_{H+\Sigma}^2 +\frac{1}{2}\|x^{k+1} - x^k\|_R^2
+\frac{1}{2\beta}\|\mu^{k+1} - \mu^k\|^2 + \frac{1}{2}
         \|x_2^{k+1} - x_2^k\|_{\beta A_2^\top A_2}^2 \nonumber \\
&  & \quad - (x_2^{k+1} - x_2^k)^\top H_{12}^\top (x_1^k - \bar x_1)  +\|x_2^{k+1} - x_2^k\|^2_{H_{22}+\Sigma_2}.
\end{eqnarray}}
Moreover, it follows from Cauchy-Schwarz inequality and $H + \Sigma \succeq 0$ that
\red{\begin{eqnarray}\label{xH-ine1}
&  &\hspace{-0.9cm} (x_2^{k+1} - x_2^k)^\top H_{12}^\top(x_1^k -
\bar x_1)  - \|x_2^{k+1} - x_2^k\|^2_{H_{22}+\Sigma_2}
 \nonumber \\
&  &\hspace{-0.9cm} = (x_2^{k+1} - x_2^k)^\top H_{12}^\top(x_1^k -
\bar x_1) + (x_2^{k+1} - x_2^k)^\top (H_{22}+\Sigma_2)(x_2^k - \bar x_2) - (x_2^{k+1} - x_2^k)^\top (H_{22}+\Sigma_2)(x_2^{k+1} - \bar x_2)
\nonumber \\
&  & \hspace{-0.9cm}= \left[
          \begin{array}{c}
           0  \\
           x_2^{k+1} - x_2^k
            \end{array}
           \right]^\top (H+\Sigma) (x^k - \bar x) - (x_2^{k+1} - x_2^k)^\top (H_{22}+\Sigma_2)(x_2^{k+1} - \bar x_2) \nonumber \\
&  & \hspace{-0.9cm}\le \frac{3}{4}\|x^k - \bar x\|_{H+\Sigma}^2 +
\frac{1}{3}\|x_2^{k+1} - x_2^k\|^2_{H_{22}+\Sigma_2}
         - \frac{1}{2}\|x_2^{k+1} - x_2^k\|^2_{H_{22}+\Sigma_2}  \nonumber \\
&  & \hspace{-0.9cm}\quad  + \frac{1}{2}\big(\|x_2^k - \bar x_2\|^2_{H_{22}+\Sigma_2} - \|x^{k+1}_2 - \bar x_2\|^2_{H_{22}+\Sigma_2}\big)\nonumber \\
&  & \hspace{-0.9cm}= \frac{3}{4}\|x^k - \bar x\|_{H+\Sigma}^2
         - \frac{1}{6}\|x_2^{k+1} - x_2^k\|^2_{H_{22}+\Sigma_2}   + \frac{1}{2}\big(\|x_2^k - \bar x_2\|^2_{H_{22}+\Sigma_2}
         - \|x^{k+1}_2 - \bar x_2\|^2_{H_{22}+\Sigma_2}\big).
\end{eqnarray}}
Using the elementary inequality $2(\|a\|_{H+\Sigma}^2 +
\|b\|_{H+\Sigma}^2) \ge \|a-b\|_{H+\Sigma}^2$, we obtain
\begin{eqnarray}\label{xH-ine2}
&  & \|x^{k+1} - \bar x\|^2_{H+\Sigma} - \frac{3}{4}\|x^k - \bar x\|_{H+\Sigma}^2 \nonumber \\
&  & = \frac{7}{8}\big(\|x^{k+1} - \bar x\|^2_{H+ \Sigma} - \|x^k -
\bar x\|_{H+\Sigma} ^2\big)
+ \frac{1}{8}\big(\|x^{k+1} - \bar x\|^2_{H+ \Sigma} + \|x^k - \bar x\|_{H+ \Sigma}^2\big) \nonumber \\
&  & \ge \frac{7}{8}\big(\|x^{k+1} - \bar x\|^2_{H+ \Sigma} - \|x^k
- \bar x\|_{H+ \Sigma}^2\big) + \frac{1}{16} \|x^{k+1} - x^k \|_{H+
\Sigma}^2.
\end{eqnarray}
Substituting \eqref{xH-ine1} and \eqref{xH-ine2} into
\eqref{contra-ine1}, we get \eqref{contraction-ine}.

\medskip

\noindent  (ii) From \eqref{contraction-ine},  we can immediately see that
\red{\begin{equation}\label{sum}
\sum_{k=1}^\infty \big(\frac{1}{16}\|x^{k+1} - x^k\|^2_{H+\Sigma+8R} +
\frac{1}{6}\|x_2^{k+1} - x_2^k\|^2_{H_{22} + \Sigma_2 + 3\beta
A_2^\top A_2}
    + \frac{1}{2\beta}\|\mu^{k+1} - \mu^k\|^2\big)<\infty,
\end{equation}}
and it therefore holds  that
\red{\begin{equation} \label{limt2} \lim_{k \rightarrow \infty}\|x^{k+1} - x^k\|_{H+\Sigma+8R} = 0, \quad
\lim_{k \rightarrow \infty}\|x_2^{k+1} - x_2^k\|_{H_{22} + \Sigma_2 + 3\beta A_2^\top A_2} = 0
 \end{equation}}
and
\begin{equation} \label{lim-mu} \lim_{k \rightarrow \infty}
 \|A_1 x_1^{k+1} + A_2 x_2^{k+1} - b\|=\lim_{k \rightarrow
\infty}{1\over \beta}\|\mu^{k+1} - \mu^k\|
    = 0.\end{equation}
Since \red{$H+\Sigma$, $R$} and $H_{22}+\Sigma_2$ are positive semidefinite matrices, we deduce from
\eqref{limt2} that
\red{\begin{equation}\label{lim-sequence}
\left\{
\begin{array}{l}
\displaystyle \lim_{k \rightarrow \infty} (H +\Sigma)\,(x^{k+1} - x^k) =0, \\[0.2cm]
\displaystyle \lim_{k \rightarrow \infty} R\,(x^{k+1} - x^k) =0, \\[0.2cm]
\displaystyle \lim_{k \rightarrow \infty} \|x_2^{k+1} - x_2^k\|_{H_{22}+\Sigma_2}=0, \\[0.2cm]
\displaystyle  \lim_{k \rightarrow \infty}\|A_2(x_2^{k+1} - x_2^k)\| = 0,
\end{array}
\right.
\end{equation}
and hence
\begin{equation}\label{lim-x1-x2}
\lim_{k \rightarrow \infty}  (H_{11}+\Sigma_1)(x_1^{k+1} - x_1^k) + H_{12}(x_2^{k+1} - x_2^k)=0.
\end{equation}
 Using the triangle inequality, we have
 \[\left\|\left[
 \begin{array}{c}
 x_1^{k+1} - x_1^k  \\
 0
 \end{array}
 \right]\right\|_{H+\Sigma}   \le  \left\|\left[
 \begin{array}{c}
 x_1^{k+1} - x_1^k \\
 x_2^{k+1} - x_2^k
 \end{array}
 \right]\right\|_{H+\Sigma} + \left\|\left[
 \begin{array}{c}
 0  \\
 x_2^{k+1} - x_2^k
 \end{array}
 \right]\right\|_{H+\Sigma},\]
 and thus from \eqref{lim-sequence}, it follows
 \[ \lim_{k \rightarrow \infty} \|x_1^{k+1} - x_1^k\|_{H_{11}+\Sigma_1} \le
 \lim_{k \rightarrow \infty} \big(\|x^{k+1} - x^k\|_{H+\Sigma} + \|x_2^{k+1} - x_2^k\|_{H_{22}+\Sigma_2}\big) =0.\]
 From \eqref{lim-sequence}, \eqref{lim-x1-x2} and the above formula, we obtain
\begin{equation}\label{lim-x1-x2-R}
\left\{
\begin{array}{l}
\displaystyle \lim_{k \rightarrow \infty} R_1(x_1^{k+1} - x_1^k) =0
, \\[0.2cm]
\displaystyle \lim_{k \rightarrow \infty} R_2(x_2^{k+1} - x_2^k) =0
, \\[0.2cm]
\displaystyle  \lim_{k \rightarrow \infty}  H_{12}(x_2^{k+1} - x_2^k) =-\lim_{k\rightarrow \infty}(H_{11}+\Sigma_1)(x_1^{k+1} - x_1^k) =0, \\[0.2cm]
\displaystyle  \lim_{k \rightarrow \infty}  A_2(x_2^{k+1} - x_2^k) = 0.
\end{array}
\right.
\end{equation} }
This, together with \eqref{opt-cond} and \eqref{lim-mu}, proves the
assertion \eqref{KKT}. \quad
\end{proof}

\medskip

To establish the convergence of ADMM, we make the following assumption:
\begin{assumption}\label{asmp:1}
We assume
    \begin{equation}\label{ADMM2}
    \left[\begin{array}{cc}
    H_{11} & 0  \\
    0 & H_{22}
    \end{array}\right] +  \left[\begin{array}{cc}
    \Sigma_1 & 0  \\
    0        & \Sigma_2
    \end{array}\right] +
    \left[\begin{array}{cc}
    A_1\zz A_1 & 0  \\
    0 & A_2\zz A_2
    \end{array}\right] \red{+  \left[\begin{array}{cc}
    R_1  & 0  \\
    0    & R_2
    \end{array}\right]}
    \succ 0.
    \end{equation}
\end{assumption}

 It is worth emphasizing that Assumption \ref{asmp:1} means that the subproblems of   $2$-block \red{proximal} ADMM admit unique solutions, because Assumption \ref{asmp:1} holds
 if and only if
    \begin{equation*}
    \left[\begin{array}{cc}
    H_{11} & 0  \\
    0 & H_{22}
    \end{array}\right] +  \left[\begin{array}{cc}
    \Sigma_1 & 0  \\
    0        & \Sigma_2
    \end{array}\right] +  \beta
    \left[\begin{array}{cc}
    A_1\zz A_1 & 0  \\
    0 & A_2\zz A_2
    \end{array}\right]\red{+  \left[\begin{array}{cc}
    	R_1  & 0  \\
    	0    & R_2
    	\end{array}\right]}
    \succ 0
    \end{equation*}
    for any $\beta>0$. However, the optimal solution to original problem \eqref{min-problem}
    is not necessarily unique.

    \red{We are now ready to} prove the iterate convergence of the $2$-block \red{proximal} ADMM for the nonseparable
    convex optimization model  \eqref{min-problem}.

\begin{theorem}\label{Conver-Alg}
Suppose Assumption {\em\ref{asmp:1}} holds. Let $\{(x_1^k, x_2^k,
\mu^k)\}$ be generated by the \red{proximal} ADMM \eqref{eq:admm} with $n=2$ to solve problem
\eqref{min-problem}. Then the sequence $\{(x_1^k, x_2^k,\mu^k)\}$
converges to a KKT point of
\eqref{min-problem}.
\end{theorem}
\begin{proof}It follows from \eqref{contraction-ine}
that the sequences \red{$\{(H+\Sigma+R)x^{k+1}\}$},
\red{$\{(H_{22}+\Sigma_2+\beta A_2^\top A_2 + R_2)x_2^{k+1}\}$}
and $\{\mu^{k+1}\}$ are all bounded.
Since \red{$H_{22} + \Sigma_2 + \beta A_2^\top A_2 + R_2$ is positive definite},
we know
$\{x_2^{k+1}\}$ is bounded. Note
that $A_1 \bar x_1 + A_2 \bar x_2 = b$.  Using
the triangle \red{inequality}
\begin{eqnarray}
\|A_1(x_1^{k+1} -  \bar x_1)\| & \le & \|A_1 x_1^{k+1} + A_2
x_2^{k+1}-(A_1\bar x_1 + A_2\bar x_2)\| + \|A_2(x_2^{k+1} - \bar
x_2)\|
\nonumber\\
&  = & \|A_1 x_1^{k+1} + A_2 x_2^{k+1} - b\| + \|A_2(x_2^{k+1} -
\bar x_2)\|\nonumber\\
&=& {1\over \beta} \|\mu^k-\mu^{k+1}\| + \|A_2(x_2^{k+1} -
\bar x_2)\| \nonumber
\end{eqnarray}
and
\red{\begin{eqnarray}
\|x_1^{k+1} - \bar x_1\|_{H_{11}+\Sigma_1 + R_1}	  & = & \left\|\left[
\begin{array}{c}
x_1^{k+1} - \bar x_1   \\
0
\end{array}
\right]\right\|_{H+\Sigma+R}\nonumber \\
	& \le &   	\left\|\left[
	\begin{array}{c}
	x_1^{k+1} - \bar x_1   \\
	x_2^{k+1} - \bar x_2
	\end{array}
	\right]\right\|_{H+\Sigma+R} +
		\left\|\left[
		\begin{array}{c}
		0  \\
		x_2^{k+1} - \bar x_2
		\end{array}
		\right]\right\|_{H+\Sigma+R}                        \nonumber \\
	& = & \|x^{k+1} - \bar x\|_{H + \Sigma + R} +  \|x_2^{k+1} - \bar x_2\|_{H_{22}+\Sigma_2 + R_2}, \nonumber\end{eqnarray} }
we further obtain the boundedness of the sequences $\{A_1x_1^{k+1}\}$
and \red{$\{(H_{11}+\Sigma_1 + R_1)x_1^{k+1}\}$},
and hence \red{$\{(H_{11} +\Sigma_1+ \beta A_1^\top A_1 + R_1)x_1^{k+1}\}$} is bounded. Together with the positive definiteness of \red{$H_{11} +
\Sigma_1 +\beta A_1^\top A_1 + R_1$}, this implies the
boundedness of $\{x_1^{k+1}\}$. Thus, the sequence $\{(x_1^k, x_2^k,
\mu^k)\}$ is bounded and there exists a triple $(x_1^{\infty},
x_2^{\infty}, \mu^{\infty})$ and a subsequence $\{k_i\}$ such that
\[
\lim_{i\rightarrow \infty} x_1^{k_i}=x_1^\infty,\qquad
\lim_{i\rightarrow \infty} x_2^{k_i}=x_2^\infty \qquad
\hbox{and} \qquad \lim_{i\rightarrow \infty} \mu^{k_i}=\mu^\infty.
\]
Setting $k =k_i-1$ and invoking the upper semicontinuity of
$\partial \theta_1$ and $\partial \theta_2$  in \eqref{KKT}, we
then obtain
\begin{equation*}
\left\{
\begin{array}{l}
- \nabla_{x_1}\phi(x_1^{\infty}, x_2^{\infty}) +  A_1^\top \mu^{\infty} \in \partial \theta_1(x_1^{\infty}), \\[0.2cm]
- \nabla_{x_2}\phi(x_1^{\infty}, x_2^{\infty}) +  A_2^\top \mu^{\infty} \in \partial \theta_2(x_2^{\infty}), \\[0.2cm]
A_1 x_1^{\infty} + A_2 x_2^{\infty} - b = 0,
\end{array}
\right.
\end{equation*}
which means $(x_1^{\infty}, x_2^{\infty}, \mu^{\infty})$ is a KKT
point of \eqref{min-problem}. Hence \eqref{contraction-ine} is also
valid if $(\bar{x}_1,{\bar x}_2, \bar{\mu})$ is replaced by
$(x_1^{\infty}, x_2^{\infty}, \mu^{\infty})$. Therefore, it holds
for any $k\geq k_i$ that
\red{\begin{eqnarray}
&  & \frac{7}{8}\|x^{k+1} - x^{\infty}\|_{H + \Sigma + \frac{4}{7}R}^2 +
\frac{1}{2}\|x_2^{k+1} -
x_2^{\infty}\|_{H_{22}+\Sigma_2+\beta A_2^\top A_2}^2 + \frac{1}{2\beta}\|\mu^{k+1} - \mu^{\infty}\|^2 +
\frac{1}{2}\|x_2^{k+1} - x_2^k\|^2_{R_2}\nn \\
&  & \le  \frac{7}{8}\|x^{k_i} - x^{\infty}\|_{H +\Sigma + \frac{4}{7}R}^2 +
\frac{1}{2}\|x_2^{k_i} - x_2^{\infty}\|_{H_{22}+\Sigma_2+\beta
A_2^\top A_2}^2 + \frac{1}{2\beta}\|\mu^{k_i} - \mu^{\infty}\|^2 +
\frac{1}{2}\|x_2^{k_i} - x_2^{k_i - 1}\|^2_{R_2}.
\label{theta-bound2}
\end{eqnarray}}
\red{It follows from \eqref{lim-x1-x2-R} that
\[  \lim_{k \rightarrow \infty} \|x_2^{k+1} - x_2^k\|_{ R_2} =0. \nn\]}
Note that
\red{\[ \lim_{i \rightarrow \infty} \big(\frac{7}{8}\|x^{k_i} - x^{\infty}\|_{H + \Sigma +\frac{4}{7}R}^2 + \frac{1}{2}\|x_2^{k_i} -
x_2^{\infty}\|_{H_{22}+\Sigma_2+\beta A_2^\top A_2}^2 +
\frac{1}{2\beta}\|\mu^{k_i} - \mu^{\infty}\|^2+
\frac{1}{2}\|x_2^{k_i} - x_2^{k_i - 1}\|^2_{R_2}\big) = 0, \nn \]}
and so we can deduce from \eqref{theta-bound2} that
\red{\begin{equation*}
\lim_{k \rightarrow \infty} \big( \frac{7}{8}\|x^{k+1} - x^{\infty}\|_{H
+  \Sigma + \frac{4}{7}R}^2 + \frac{1}{2}\|x_2^{k+1} -
x_2^{\infty}\|_{H_{22}+\Sigma_2+\beta A_2^\top A_2}^2 +
\frac{1}{2\beta}\|\mu^{k+1} - \mu^{\infty}\|^2+
\frac{1}{2}\|x_2^{k+1} - x_2^k\|^2_{R_2}\big) = 0,
\end{equation*}}
which implies
\[
\red{\lim_{k \rightarrow \infty}\|x_2^{k+1}
-x_2^{\infty}\|_{H_{22}+\Sigma_2+\beta A_2^\top A_2 + R_2}^2 =0,}\qquad
\lim_{k \rightarrow \infty} \mu^{k+1} =\mu^\infty
\]
and
\red{\begin{equation}\label{eq:H+sigmax}
\lim_{k \rightarrow \infty} \|x^{k+1} - x^{\infty}\|_{H+\Sigma+R}^2=0.
\end{equation}}
%Therefore, it holds that
%\[  \|x_2^{k+1} - x_2^{\infty}\|_{H_{22}+\Sigma_2 + \beta A_2^\top A_2}^2 \le  \|x^{k+1} - x^{\infty}\|_{H+\Sigma}^2 + \|x_2^{k+1} - x_2^{\infty}\|^2_{H_{22}+\beta A_2^\top A_2}
%  \rightarrow 0, \nn \]
%as $k\rightarrow \infty$.
Since \red{$H_{22}+\Sigma_2+\beta A_2^\top A_2 + R_2$} is positive definite, we
obtain
\begin{equation}\label{eq:limitx2}
\displaystyle \lim_{k\rightarrow \infty} x_2^{k+1} =x_2^\infty.
\end{equation}
On the other hand, by \eqref{KKT} and \eqref{eq:limitx2},
it can easily be seen that
\begin{eqnarray}\label{eq:limitA1x1}
\|A_1(x_1^{k+1} -  x_1^{\infty})\| & \le & \|A_1x_1^{k+1} +
A_2x_2^{k+1} - (A_1x_1^{\infty} + A_2x_2^{\infty})\| +
   \|A_2(x_2^{k+1} - x_2^{\infty})\| \nonumber
\\[5pt]
  & = & \|A_1x_1^{k+1} + A_2x_2^{k+1} - b\| + \|A_2(x_2^{k+1} -  x_2^{\infty})\| \rightarrow 0,
\end{eqnarray}
as $k\rightarrow \infty$. Then, we obtain
\red{\begin{eqnarray}
& & \|x_1^{k+1} - x_1^{\infty}\|_{H_{11}+\Sigma_1 + \beta A_1^\top
A_1 + R_1}^2  =  \|x_1^{k+1} - x_1^{\infty}\|_{H_{11} + \Sigma_1 + R_1}^2 + \beta
\|A_1(x_1^{k+1} -  x_1^{\infty})\|^2  \nonumber \\ [0.2cm] & & \;\; =  \left\|\left[
\begin{array}{c}
x_1^{k+1} - x_1^{\infty}   \\
0
\end{array}
\right]\right\|_{H+\Sigma+R}^2 +   \beta
\|A_1(x_1^{k+1} -  x_1^{\infty})\|^2 \nonumber\\[0.2cm]
 & & \;\; \le \left( \left\|\left[
 \begin{array}{c}
 x_1^{k+1} - x_1^{\infty}   \\
 x_2^{k+1} - x_2^{\infty}
 \end{array}
 \right]\right\|_{H+\Sigma+R} +  \left\|\left[
 \begin{array}{c}
  0   \\
 x_2^{k+1} - x_2^{\infty}
 \end{array}
 \right]\right\|_{H+\Sigma+R}\right)^2 +   \beta
 \|A_1(x_1^{k+1} -  x_1^{\infty})\|^2 \nonumber\\[0.2cm]
 && \;\;=
(\|x^{k+1} - x^{\infty}\|_{H+\Sigma+R} + \|x_2^{k+1} -
x_2^{\infty}\|_{H_{22}+\Sigma_2+R_2})^2+ \beta \|A_1(x_1^{k+1} -
x_1^{\infty})\|^2,  \nonumber
\end{eqnarray}}
where  ``$\leq$'' follows the triangle inequality of norms. Together with \eqref{eq:H+sigmax}, \eqref{eq:limitx2},
\eqref{eq:limitA1x1}, and the positive definiteness of
\red{$H_{11}+\Sigma_1+\beta A_1^\top A_1 + R_1$}, this shows that
\[
\lim_{k \rightarrow \infty} x_1^{k+1} = x_1^{\infty}.
\]
Therefore, we have shown that the whole sequence $\{(x_1^k, x_2^k,
\mu^k)\}$ converges to $(x_1^{\infty}, x_2^{\infty}, \mu^{\infty})$,
which is a KKT
point of \eqref{min-problem}.
This comletes the proof.\quad
\end{proof}

    \begin{remark} In fact,
    the iterate convergence of $2$-block \red{proximal} ADMM can also be guaranteed
    if there is a fixed stepsize $\gamma\in(0, (1+\sqrt{5})/2)$ in the dual update. Namely,
    the \red{proximal} ADMM can be  extended as follows:
\red{\begin{equation}\label{Iter-ADMM-extend-gamma}
\left\{
\begin{array}{l}
  \displaystyle x_1^{k+1} := \argmin_{x_1 \in \R^{d_1}} \Big\{ {\cal L}_{\beta}(x_1, x_2^k; \mu^k) + \frac{1}{2}\|x_1 - x_1^k\|_{R_1}^2\Big\}, \\[5pt]
  \displaystyle x_2^{k+1} := \argmin_{x_2 \in \R^{d_2}}\Big\{{\cal L}_{\beta}(x_1^{k+1}, x_2; \mu^k)) + \frac{1}{2}\|x_2 - x_2^k\|_{R_2}^2 \Big\}, \\[5pt]
  \mu^{k+1}: = \mu^k - \gamma\beta (A_1x_1^{k+1} + A_2 x_2^{k+1} - b),
\end{array}
\right.
\end{equation}}
where $\beta > 0$ and $\gamma \in (0, (1+\sqrt{5})/2)$. Under the
 conditions of Theorem {\em \ref{Conver-Alg}}, we can similarly
prove the global iterate convergence of \eqref{Iter-ADMM-extend-gamma}. For
brevity, we omit the details here.
\end{remark}

\red{
 \begin{remark} The proximal ADMM includes the ADMM and its linearized version as special cases.
When $R_1=0$ and $R_2=0$, the proximal ADMM reduces to the ADMM and, according to Theorem
\ref{Conver-Alg}, its convergence can be established under the condition that
    \begin{equation*}
    \left[\begin{array}{cc}
    H_{11} & 0  \\
    0 & H_{22}
    \end{array}\right] +  \left[\begin{array}{cc}
    \Sigma_1 & 0  \\
    0        & \Sigma_2
    \end{array}\right] +
    \left[\begin{array}{cc}
    A_1\zz A_1 & 0  \\
    0 & A_2\zz A_2
    \end{array}\right]
    \succ 0
    \end{equation*}
The ADMM can be easily applied to the convex minimization problems where $\theta_i\,(i=1,2)$ have closed form proximal
operators and all the matrices $H_{11}, H_{22}, A_1, A_2$ are diagonal.  Otherwise, we  consider the linearized ADMM:
\begin{eqnarray*}
\left\{
\begin{array}{l}
\displaystyle {x_1^{k+1/2}} : =\big[{(r_1 I - H_{11}-\beta A_1^TA_1})x_1^k-\beta A_1^T(A_2x_2^k-b) -H_{12}x_2^k-g_1\big]/r_1\\[0.2cm]
\displaystyle  x_1^{k+1} := \argmin\limits_{x_1 \in \R^{d_1}} \theta_1(x_1)+ \frac{r_1}{2}\|x_1 - x_1^{k+1/2}\|^2,\\
\displaystyle {x_2^{k+1/2}} : =\big[({r_2 I - H_{22}-\beta A_2^TA_2})x_2^k-\beta A_2^T(A_1x_1^{k+1}-b) -H_{12}^Tx_1^{k+1}-g_2\big]/r_2\\[0.2cm]
\displaystyle  x_2^{k+1} := \argmin\limits_{x_2 \in \R^{d_2}} \theta_2(x_2)+ \frac{r_2}{2}\|x_2 - x_2^{k+1/2}\|^2,\\
\displaystyle \mu^{k+1} := \mu^k -\beta (A_1x_1^{k+1}+A_2x_2^{k+1}-b),
\end{array}
\right.
\end{eqnarray*}
which is equivalent to the proximal ADMM with $R_1 =r_1 I- H_{11}-\beta A_1^TA_1$
and $R_2 =r_2 I- H_{22}-\beta A_2^TA_2$.  Thus the iterate convergence of linearized ADMM can be guaranteed
under the condition that
\[
r_i \geq \max_{1\leq j \leq d_i} \lambda_j (H_{ii}+\beta A_i^TA_i),\quad i=1,2,
\]
where $\lambda_j(\cdot)$ represents the $j$th eigenvalue of a matrix.

\end{remark}
}

\blue{By using the following proposition (see \cite[Lemma 1.1]{Deng2016} and \cite[Lemma 3]{lisuntoh2016SIOPT}),
we can deliver a $o(1/k)$ convergence rate of the proximal ADMM, measured by
the square of KKT violation.
\begin{proposition}\label{prop-lem}
	\magenta{For} any sequence  $\{a_i\} \subseteq \Re$ satisfying
	$a_i \ge 0$ and $\sum_{i=1}^{\infty} a_i < + \infty$, it holds that
	$\min_{1 \le i \le k}\{a_i\}= o(1/k)$.
\end{proposition}
}

\red{\begin{theorem}\label{Conver-Alg-1}
	Suppose Assumption {\em\ref{asmp:1}} holds. Let $\{(x_1^k, x_2^k,
	\mu^k)\}$ be generated by the \red{proximal} ADMM \eqref{eq:admm} with $n=2$ to solve problem
	\eqref{min-problem}. Then, we have
	\begin{eqnarray}\label{iter-complex}
& &  \min_{1 \le i \le k}\Big\{d^2\big(0, \; \partial \theta_1(x_1^{i+1}) + \nabla_{x_1}\phi(x_1^{i+1}, x_2^{i+1}) - A_1^\top \mu^{i+1}\big) + d^2\big(0, \; \partial \theta_2(x_2^{i+1}) + \nabla_{x_2}\phi(x_1^{i+1}, x_2^{i+1}) - A_2^\top \mu^{i+1}\big)  \nonumber\\
&  & \qquad\quad + \|A_1 x_1^{i+1} + A_2 x_2^{i+1} - b\|^2\Big\} = o(1/k).
	\end{eqnarray}
\end{theorem}
\begin{proof} From \eqref{opt-cond} and \eqref{eq:admm}, we obtain
\begin{equation*}
\left\{
\begin{array}{l}
\displaystyle -R_1(x_1^{k+1} - x_1^k) + (H_{12} + \beta A_1^\top A_2)(x_2^{k+1} - x_2^k) \in  \partial \theta_1(x_1^{k+1}) + \nabla_{x_1}\phi(x_1^{k+1}, x_2^{k+1}) - A_1^\top \mu^{k+1}, \\[5pt]
\displaystyle  -R_2(x_2^{k+1} - x_2^k)  \in  \partial \theta_2(x_2^{k+1}) + \nabla_{x_2}\phi(x_1^{k+1}, x_2^{k+1}) - A_2^\top \mu^{k+1}, \\[5pt]
\displaystyle  \|A_1 x_1^{k+1} + A_2 x_2^{k+1} - b\|^2 = \frac{1}{\beta^2}\|\mu^{k+1} - \mu^k\|^2.
\end{array}
\right.
\end{equation*}
By using the Cauchy-Schwarz inequality and the above formulas, we obtain
 \begin{eqnarray}\label{sum-1}
&  & d^2\big(0, \; \partial \theta_1(x_1^{k+1}) + \nabla_{x_1}\phi(x_1^{k+1}, x_2^{k+1}) - A_1^\top \mu^{k+1}\big) + d^2\big(0, \; \partial \theta_2(x_2^{k+1}) + \nabla_{x_2}\phi(x_1^{k+1}, x_2^{k+1}) - A_2^\top \mu^{k+1}\big)  \nonumber\\
&  & \qquad\quad + \|A_1 x_1^{k+1} + A_2 x_2^{k+1} - b\|^2 \nonumber \\
&  & \quad \le  2\|R_1(x_1^{k+1} - x_1^k)\|^2 + 2
\|(H_{12} + \beta A_1^\top A_2)(x_2^{k+1} - x_2^k)\|^2
+ \|R_2(x_2^{k+1} - x_2^k)\|^2 + \frac{1}{\beta^2}\|\mu^{k+1} - \mu^k\|^2 \nonumber\\
&  & \quad \le 2\|R_1^{\frac{1}{2}}\|^2\|x_1^{k+1} - x_1^k\|^2_{R_1} +(2\|H_{12}+\beta A_1^\top A_2\|^2 + \|R_2\|^2)\|x_2^{k+1} - x_2^k\|^2 + \frac{1}{\beta^2}\|\mu^{k+1} - \mu^k\|^2.
 \end{eqnarray}
It follows from \eqref{sum} that
\begin{equation}\label{Iter-sum}
\left\{
\begin{array}{l}
\displaystyle\sum_{k=1}^{\infty}\|x_1^{k+1} - x_1^k\|^2_{R_1} < \infty, \\[5pt]
\displaystyle \sum_{k=1}^{\infty}\|x_2^{k+1} - x_2^k\|^2_{H_{22}+\Sigma_2+  A_2^\top A_2+R_2} < \infty, \\[5pt]
\displaystyle\sum_{k=1}^{\infty}\|\mu^{k+1} - \mu^k\|^2< \infty,
\end{array}
\right.
\end{equation}
Since $H_{22}+\Sigma_2+  A_2^\top A_2+R_2 \succ 0$, from \eqref{Iter-sum} we
have
\begin{eqnarray}\label{eq:newadd}
\sum_{k=1}^{\infty}\|x_2^{k+1} - x_2^k\|^2 < \infty.
\end{eqnarray}
\blue{Combining Proposition \ref{prop-lem}
	 with the relationships \eqref{Iter-sum} and \eqref{eq:newadd}, we have
\[  \min_{1 \le i \le k}\Big\{2\|R_1^{\frac{1}{2}}\|^2\|x_1^{i+1} - x_1^i\|^2_{R_1} +(2\|H_{12}+\beta A_1^\top A_2\|^2 + \|R_2\|^2)\|x_2^{i+1} - x_2^i\|^2 + \frac{1}{\beta^2}\|\mu^{i+1} - \mu^i\|^2\Big\} = o(1/k),\]
which, together with \eqref{sum-1}, implies \eqref{iter-complex}. We complete the proof.}
\end{proof}
}

We remark that, in some sense,  Assumption \ref{asmp:1} actually acts as the
weakest condition to guarantee the iterate convergence of the \red{proximal} ADMM for solving problem
\eqref{min-problem}. Firstly, if Assumption \ref{asmp:1} is violated, the solution sets of
subproblems in
 \eqref{eq:admm} might be empty, in which case the $2$-block \red{proximal} ADMM scheme is not well defined
 (see \cite{ChenSunToh2015-2} for an illustration).
  Secondly, the following corollary shows that Assumption  \ref{asmp:1}
is not only sufficient, but also necessary for the iterate convergence of the  $2$-block \red{proximal}
 ADMM for solving the coupled quadratic minimization problem.
Thus, the conditions we proposed are
 already tight.

\begin{corollary}\label{Conver-Alg2-Cor}
Assume problem \eqref{min-problem} is a convex
quadratic programming problem, that is $\theta_1(x_1) \equiv 0$ and
$\theta_2(x_2) \equiv 0$. Then, any sequence generated by the
$2$-block \red{proximal} ADMM is convergent if and only if  Assumption
{\em\ref{asmp:1}} holds.
%
%Let $\{(x_1^k, x_2^k, \mu^k)\}$ be generated by the $2$-block
%ADMM. Then, the sequence $\{(x_1^k, x_2^k,\mu^k)\}$ converges to
%some primal-dual solution pair of the problem \eqref{min-problem} if
%and only if the condition \eqref{ADMM2-1} holds.
\end{corollary}
\begin{proof} The ``if" part follows immediately from Theorem \ref{Conver-Alg}.
 For the ``only if" part, we prove that if Assumption \ref{asmp:1}
 fails to hold, there must exist some sequence generated
by the 2-block \red{proximal} ADMM that is divergent.  Indeed, let
$\{(x_1^k,x_2^k,\mu^k)\}$ be a sequence generated by the $2$-block \red{proximal} ADMM, i.e.,
\red{\begin{equation}\label{Iter-ADMM-part1}
\left\{
\begin{array}{l}
  \displaystyle x_1^{k+1}  \in \argmin_{x_1 \in \R^{d_1}} \Big\{{\cal L}_{\beta}(x_1, x_2^k; \mu^k) + \frac{1}{2}\|x_1 - x_1^k\|^2_{R_1}\Big\}, \\[5pt]
  \displaystyle x_2^{k+1}  \in \argmin_{x_2 \in \R^{d_2}} \Big\{{\cal L}_{\beta}(x_1^{k+1}, x_2; \mu^k) + \frac{1}{2}\|x_2 - x_2^k\|^2_{R_2}\Big\}, \\[5pt]
  \mu^{k+1}=\mu^k -\beta (A_1x_1^{k+1} + A_2x_2^{k+1}-b).
\end{array}
\right.
\end{equation}}
If the sequence is divergent, then the ``only if'' part of this
corollary holds. Thus we  need only consider the case where
$\{(x_1^k,x_2^k,\mu^k)\}$ converges. Because  \red{$H_{ii}+\beta A_i^\top
A_i+R_i$}\, $(i=1,2)$ are not positive definite, there exists a nonzero
vector $(\bar y_1, \bar y_2)$ such that
\red{\[
(H_{ii}+\beta A_i^\top A_i + R_i) \bar y_i =0 \quad \forall\, i=1,2,
\]}
or equivalently,
\red{\begin{equation}\label{eq:Hequation1}
H_{ii} \bar y_i=0, \qquad A_i
\bar{y}_i=0 \qquad {\rm and}\qquad  R_i
\bar{y}_i=0 \,\,\quad \forall\, i=1,2.
\end{equation}}
Using the fact that
$
0
\preceq  H  \preceq 2
\left[\begin{array}{cc}
    H_{11} & 0  \\
    0 & H_{22}
    \end{array}\right]
$,
we have $H\bar y=0$. Hence, it holds that
\begin{equation}\label{eq:Hequation2}
H_{12} \bar y_2=0 \qquad \quad{\rm and} \qquad \quad H_{12}^\top
\bar y_1=0.
\end{equation}
By \eqref{Iter-ADMM-part1}, \eqref{eq:Hequation1}
and \eqref{eq:Hequation2}, it  can easily be seen that, for any $k\geq 1$,
\red{
\begin{equation*}
\left\{
\begin{array}{l}
  \displaystyle  x_1^{2k} +\bar y_1 \in \argmin_{x_1} \Big\{{\cal L}_{\beta}(x_1, x_2^{2k-1}; \mu^{2k-1}) + \frac{1}{2}\|x_1 - x_1^{2k-1}\|^2_{R_1}\Big\}, \\[5pt]
  \displaystyle x_2^{2k} + \bar y_2 \in \argmin_{x_2}  \Big\{{\cal L}_{\beta}(x_1^{2k}+\bar y_1, x_2; \mu^{2k-1}) + \frac{1}{2}\|x_2 - x_2^{2k-1}\|^2_{R_2}\Big\}, \\[5pt]
  \mu^{2k}=\mu^{2k-1} -\beta \big(A_1( x_1^{2k} +\bar y_1) + A_2 (x_2^{2k}+\bar y_2) -b \big)
\end{array}
\right.
\end{equation*}}
and
\red{\begin{equation*}
\left\{
\begin{array}{l}
  \displaystyle  x_1^{2k+1} \in \argmin_{x_1} \Big\{ {\cal L}_{\beta}(x_1, x_2^{2k}+\bar y_2; \mu^{2k}) + \frac{1}{2}\|x_1 - (x_1^{2k}+\bar y_1)\|^2_{R_1}\Big\}, \\[5pt]
  \displaystyle x_2^{2k+1}  \in \argmin_{x_2}  \Big\{{\cal L}_{\beta}(x_1^{2k+1}, x_2; \mu^{2k})  + \frac{1}{2}\|x_2 - (x_2^{2k}+\bar y_2)\|^2_{R_2}\Big\}, \\[5pt]
  \mu^{2k+1}=\mu^{2k} -\beta \big(A_1 x_1^{2k+1}  + A_2 x_2^{2k+1} -b
  \big).
\end{array}
\right.
\end{equation*}}
This means that the divergent sequence $( x_1^1, x_2^1,
\mu^1)\rightarrow(x_1^2 + \bar y_1, x_2^2 + \bar y_2, \mu^2)
\rightarrow ( x_1^3, x_2^3,
\mu^3) \rightarrow(x_1^4 + \bar y_1, x_2^4 + \bar y_2, \mu^4)
\rightarrow \ldots$ could
be generated by the $2$-block \red{proximal} ADMM. Thus,  Assumption \ref{asmp:1} is also necessary for the
iterate convergence. This completes the proof. \quad
\end{proof}

When restricted to the case that $A_i\, (i=1,2)$ and $b$
are absent, the $2$-block \red{proximal} ADMM reduces to the $2$-block cyclic proximal BCD method. Our analysis of proximal ADMM provides an iterate convergence result for the $2$-block cyclic proximal BCD method
without assuming %the existence of accumulation points of the iteration sequence
the boundedness of the iterates,
but only requiring a condition to ensure the uniqueness of the subproblem solutions.
This result is an important supplement to traditional studies on BCD,
which have mainly focused on subsequence convergence and the complexity of the function values, and enables a better understanding of the performance of this method.

 \begin{corollary}\label{Conver-Alg2}
Assume \red{$H_{ii}+\Sigma_i + R_i
\succ 0$ }. Let $\{(x_1^k, x_2^k)\}$ be generated by the cyclic proximal
BCD \eqref{eq:BCD} with $n=2$ to solve the following unconstrained optimization problem:
\begin{equation}\label{prob:un}
\displaystyle \min_{x\in \R^d} \,\,\theta_1(x_1) + \theta_2(x_2) +  {1\over 2}x^\top Hx + g^\top x.\\[5pt]
\end{equation}
Then the whole sequence $\{(x_1^k, x_2^k)\}$  converges to an optimal solution of
\eqref{prob:un}.
\end{corollary}

\red{
 \begin{remark} Similar to the proximal ADMM, the proximal BCD includes BCD and its linearized version (also know as BCPG) as special cases.
When $R_1=0$ and $R_2=0$, the proximal BCD reduces to BCD and, according to Theorem
\ref{Conver-Alg}, its convergence can be established under the condition that
    \begin{equation*}
    \left[\begin{array}{cc}
    H_{11} & 0  \\
    0 & H_{22}
    \end{array}\right] +  \left[\begin{array}{cc}
    \Sigma_1 & 0  \\
    0        & \Sigma_2
    \end{array}\right]
    \succ 0
    \end{equation*}
  The BCPG is a combination of the proximal gradient method and BCD, which can be easily implemented when
  $\theta_i$ have closed-form proximal operators. Specifically, it takes the form that
\begin{eqnarray*}
\left\{
\begin{array}{l}
\displaystyle {x_1^{k+1/2}} : =\big[(r_1 I - H_{11})x_1^k -H_{12}x_2^k-g_1\big]/r_1\\[0.2cm]
\displaystyle  x_1^{k+1} := \argmin\limits_{x_1 \in \R^{d_1}} \theta_1(x_1)+ \frac{r_1}{2}\|x_1 - x_1^{k+1/2}\|^2,\\
\displaystyle {x_2^{k+1/2}} : =\big[(r_2 I - H_{22})x_2^k-H_{12}^Tx_1^{k+1}-g_2\big]/r_2\\[0.2cm]
\displaystyle  x_2^{k+1} := \argmin\limits_{x_2 \in \R^{d_2}} \theta_2(x_2)+ \frac{r_2}{2}\|x_2 - x_2^{k+1/2}\|^2,\\
\end{array}
\right.
\end{eqnarray*}
which is equivalent to the proximal BCD with $R_1 =r_1 I- H_{11}$
and $R_2 =r_2 I- H_{22}$.  Thus the iterate convergence of linearized ADMM can be guranteed
under the condition that
\[
r_i \geq \max_{1\leq j \leq d_i} \lambda_j (H_{ii}),\quad i=1,2.
\]
\end{remark}
}

%Despite the fact that various sublinear convergence rates have been established for BCD-type methods
%\cite{Beck2013,RHLuo2013,Beck2015,shefi15},  none can be directly applied to our 2-block cyclic proximal BCD for problem \eqref{prob:un}. This is because its level set is not necessarily bounded, which violates the common assumption in the above-mentioned analysis. Note that, under the condition in
%Corollary \ref{Conver-Alg2}, the sequence generated by the BCD method is bounded. Thus, we can also obtain the $O(1/k)$ global convergence rate of the method  using the main techniques developed in \cite{Beck2015}.
%Specifically, we assume
%\red{\[
%\|x_1^k-\bar{x}_1\|\leq \bar R_1 \quad{\rm and}\quad
%\|x_2^k - \bar{x}_2\| \leq \bar R_2,
%\]}
%where $(\bar{x}_1,\bar{x}_2)$ is an optimal solution  of \eqref{prob:un}.
%Let $L_1$ and $L_2$ be the largest eigenvalues of $H_{11}$ and $H_{22}$, respectively.
%By invoking  \cite[Theorem 3.9]{Beck2015}, we immediately have the following corollary.
%
% \begin{corollary}\label{Conver-Alg3}
%Assume  $H_{ii}+\Sigma_i\succ 0$ for $i=1,2$. Let $\{(x_1^k, x_2^k)\}_{k\geq 0}$ be the sequence generated by the
% cyclic BCD \eqref{eq:BCD} with $n=2$ and $R_i=0\,(i=1,2)$ to solve problem \eqref{prob:un}. Then, it holds that
%\red{\[
%\theta(x^k) -\theta^* \leq \max \left\{\bigg({1\over 2}\bigg)^{k-1\over 2}
%(\theta(x^0)-\theta^*),{8\min\{L_1,L_2\}(\bar R_1^2 + \bar R_2^2)\over k-1}
%\right\}
%\]}
%where $\theta^*$ is the optimal function value of \eqref{prob:un}.
%\end{corollary}

\section{Convergence of Multi-block RPADMM and RPBCD}

As shown in \cite{CHYY2013}, the convergence result for $2$-block
ADMM obtained in the previous section cannot be extended to the multi-block case, i.e., $n\geq 3$. To remove the possibility of divergence, we use randomly permuted ADMM (RPADMM) to solve the nonseparable optimization problem
\eqref{eq:convobj}. Specifically, RPADMM first picks a permutation $\sigma$ of $\{1, \dots, n \}$  uniformly at random,
and then iterates as follows:
\begin{eqnarray}\label{eq:rpadmm}
\left\{
\begin{array}{l}
\displaystyle x_{\sigma(1)}^{k+1} := \argmin\limits_{x_{\sigma(1)}} \mathcal{L}_\beta (x_{\sigma(1)}, x_{\sigma(2)}^k, \ldots, x_{\sigma(n)}^k;\mu^k),\\
\displaystyle x_{\sigma(2)}^{k+1} := \argmin\limits_{x_{\sigma(2)}} \mathcal{L}_\beta (x_{\sigma(1)}^{k+1}, x_{\sigma(2)}, x_{\sigma(3)}^k, \ldots, x_{\sigma(n)}^k;\mu^k),\\
\displaystyle   \cdots\cdots  \\
\displaystyle x_{\sigma(n)}^{k+1} := \argmin\limits_{x_{\sigma(n)}} \mathcal{L}_\beta (x_{\sigma(1)}^{k+1}, x_{\sigma(2)}^{k+1},\ldots, x_{\sigma(n-1)}^{k+1}, x_{\sigma(n)};\mu^k),\\
\displaystyle \mu^{k+1} := \mu^k -\beta (\sum_{i=1}^n
A_ix_i^{k+1}-b),
\end{array}
\right.
\end{eqnarray}
where the permuted augmented Lagrangian function $\mathcal{L}_\beta (x_{\sigma(1)}, x_{\sigma(2)}, \ldots, x_{\sigma(n)};\mu)$ is defined by
\[
\mathcal{L}_\beta (x_{\sigma(1)}, x_{\sigma(2)}, \ldots, x_{\sigma(n)};\mu):=
\mathcal{L}_\beta (x_1, x_2, \ldots, x_n;\mu).
\]

It has been  shown  \cite{SLY2015} that  RPADMM is convergent in
expectation for solving the nonsingular square system of linear
equations. To extend their result to the nonseparable convex
optimization model \eqref{eq:convobj}, it is natural to first study
whether  RPADMM is even convergent in expectation for solving the following simpler linearly constrained quadratic minimization problem
\begin{eqnarray}\label{eq:quadratic}
\begin{array}{ll}
\min\limits_{x \in \R^d} &  \displaystyle \theta(x): =  \frac{1}{2}x^\top H x + g^\top x\\ [0.2cm]
\mbox{s.t.} &  {\displaystyle  \sum_{i=1}^n A_ix_i  =b,}
\end{array}
\end{eqnarray}
where $H$ can  be partitioned into $n\times n$ blocks
\red{$H_{ij}\in\R^{d_i\times d_j}$ ($1\leq i,j\leq n$)} accordingly.
In this section, we provide  an affirmative answer to the above question under the
following assumption.
\begin{assumption}\label{asmp:4}
   Assume %\remove{$H\succeq 0$, and}
    $$\left[\begin{array}{cccc}
    H_{11} & 0 & \cdots & 0 \\
    0 & H_{22}  & \cdots & 0 \\
    \vdots & \vdots & \ddots & \vdots\\
    0 & 0 & \cdots & H_{nn}\\
    \end{array}\right] +
    \left[\begin{array}{cccc}
    A_1\zz A_1 & 0 & \cdots & 0 \\
    0 & A_2\zz A_2 & \cdots & 0 \\
    \vdots & \vdots & \ddots & \vdots\\
    0 & 0 & \cdots & A_n\zz A_n\\
    \end{array}\right]
    \succ 0.$$
\end{assumption}
Although our current result is restricted
for nonseparable quadratic minimization, a special case of \eqref{eq:convobj}, it serves as a good indicator of the expected convergence of RPADMM in more general cases. It is noteworthy that our result is a non-trivial extension of the result in \cite{SLY2015}, because, in our setting, the problem under consideration is more general. For \red{example, the} optimal solution set of  \eqref{eq:quadratic} is not necessarily a singleton, in which case the spectral radius of the algorithm mapping may not be strictly less than 1, although this fact played a key role in establishing their result.

\subsection{Proof Outline and Preliminaries}

For convenience, we follow the notation in \cite{SLY2015}, and describe
the iterative scheme of RPADMM in a matrix form. Let $L_{\sigma}\in
\R^{d\times d}$ be an $n\times n$ block matrix  defined by\blue{
\begin{equation*}\label{eq:definels}
(L_{\sigma})_{\sigma(i),\sigma(j)}:=\left\{
\begin{array}{ll}
H_{\sigma(i)\sigma(j)} + \beta A_{\sigma(i)}^\top A_{\sigma(j)}, & \hbox{if} \;\; i\geq j,\\
0, & \hbox{otherwise},
\end{array}
\right.
\end{equation*}}
and $R_\sigma$ be defined as
\begin{equation}\label{eq:defineR}
R_\sigma:=L_\sigma-(H+\beta A\zz A):=L_\sigma- S.
\end{equation}
By setting $z:=(x, \mu)$, the randomly permuted ADMM can be viewed as a fix point iteration
\begin{eqnarray}\label{eq:RPADMM}
z^{k+1} := M_\sigma z^k + \bar{L}_{\sigma}\inv\bar{b},
\end{eqnarray}
where
\begin{eqnarray*}
    M_\sigma:=\bar{L}_{\sigma}\inv\bar{R}_{\sigma} ,\quad
    \bar{L}_\sigma
    :=
    \left[
    \begin{array}{cc}
        L_\sigma & 0 \\
        \beta A & I
    \end{array}
    \right],\quad
    \bar{R}_\sigma
    :=
    \left[
    \begin{array}{cc}
        R_\sigma & A\zz \\
        0 & I
    \end{array}
    \right],\quad
    \bar{b}
    :=
    \left[
    \begin{array}{c}
        -g + \beta A\zz b \\
        \beta b
    \end{array}
    \right].
\end{eqnarray*}
Define the matrix $Q$ by
    \begin{eqnarray}\label{eq:Qdef}
    Q:= E_\sigma (L_\sigma\inv) = \frac{1}{n!}\sum\limits_{\sigma\in\Gamma} L_\sigma\inv
    \end{eqnarray}
and $M$ by
\begin{equation}\label {eq:defineM}
M:= E_\sigma (M_\sigma) ={1\over n!} \sum_{\sigma\in \Gamma} M_\sigma,
\end{equation}
\medskip
where $\Gamma$ is the set of all permutations of $\{1,2,\ldots,n\}$. By direct computation, we can easily see that
    \begin{eqnarray}\label{eq:Mdef}
    M:=\left[
    \begin{array}{cc}
    I-QS & QA\zz\\
    -\beta A+\beta AQS & I- \beta AQA\zz
    \end{array}
    \right].
    \end{eqnarray}

To prove the expected convergence of the RPADMM \eqref{eq:rpadmm} for problem \eqref{eq:quadratic} under  Assumption \ref{asmp:4}, we will use a similar, but not identical, \blue{structure as that introduced in  \cite{SLY2015}, which consists of the following main steps:}
\begin{itemize}
    \item[(1)] $\mbox{eig}(QS)\subset \left [0,\frac{4}{3}\right)$;
    \item[(2)] For any eigenvalue  $\lambda$  of $M$, $\mbox{eig}(QS)\subset \left [0,\frac{4}{3}\right)$ implies that $|\lambda|<1$ or $\lambda=1$;
    \item[(3)] If $1$ is an eigenvalue of $M$, then the eigenvalue $1$ has a complete set of
    eigenvectors;
    \item[(4)] Items (2) and (3) imply the convergence in expectation of the RPADMM.
\end{itemize}

To prove the above items, we need the following \blue{linear algebra} lemmas, whose proofs can be found in the Appendix.

\begin{lemma}\label{lemma:eigQS}
    Suppose that Assumption {\em \ref{asmp:4}} holds,  $S\in\R^{d\times d}$ is a symmetric matrix defined by \eqref{eq:defineR} \red{and} $Q$ is defined by \eqref{eq:Qdef}. Then, the matrix $Q$ is positive definite  and all the
    eigenvalues of $QS$ lie in $\left[0,\frac{4}{3}\right)$, i.e.,
    \begin{eqnarray}\label{eq:eigQS}
    \mbox{\em eig}(QS)\subset \left[0,\frac{4}{3}\right).
    \end{eqnarray}
\end{lemma}
\medskip

\begin{lemma}\label{lemma:alg}
    Let $S$ and $T$ be two symmetric positive semidefinite matrices in $\R^{d\times d}$. Then, there exists a polynomial $p(x)$ such that
    \[
    {\rm det} \big((\lambda-1)^2 I + (2\lambda-1) S + (\lambda-1) T \big) = (\lambda-1)^l p(\lambda)
    \]
    and $p(1)>0$,
    where ${\rm det}(\cdot)$ denotes the determinant of some matrix,
    $l = 2d -{\rm Rank}(S) -{\rm Rank}(S+T)$ and ${\rm Rank}(\cdot)$ denotes the rank of some matrix.
\end{lemma}
\medskip

\begin{lemma}\label{lemma:geo}
    Suppose $S\in\R^{d\times d}$ is a symmetric matrix defined by \eqref{eq:defineR}
     and $\beta>0$, then
    \[
    {\rm Rank}
    \left[
    \begin{array}{cc}
    S & -A\zz \\
    \beta A & 0
    \end{array}
    \right]
    = {\rm Rank}(S) +{\rm Rank}(\beta A^\top A).
    \]
\end{lemma}
\medskip

\blue{Here, Lemma \ref{lemma:eigQS}, Step (1) of the proof structure, is an enhanced version of Lemma 2 in
\cite{SLY2015} that is compatible with problem \eqref{eq:quadratic}.} The
proofs of Steps (2) and (3),  which reveal the essential nature of this extension and are hence the key contributions here, will be presented in Subsection \ref{sec:conv-2}.
The proof for Step (4) is given in Subsection \ref{sec:conv-3}.
%will follow the technique in \cite{SLY2015}
%to give the proof for Step (3) in the last subsection.

\subsection{Eigenvalues of the Expected Update Matrix}\label{sec:conv-2}

One of the  main differences between the nonsingular linear system
case and that of the extended case is reflected in the following
lemma, where $1$ can be an eigenvalue of the expected update matrix
$M$.

\begin{lemma}\label{lemma:spectral}
    Suppose that Assumption {\em \ref{asmp:4}} holds %\ref{asmp:1} and \ref{asmp:2} hold,
    and $S\in\R^{d\times d}$ is a symmetric matrix defined by \eqref{eq:defineR}.
     Let $\lambda$ be any eigenvalue of $M$,
    then we have either $|\lambda|<1$ or $\lambda=1$.
\end{lemma}
\begin{proof}\blue{
    We introduce the following notation:
    \begin{eqnarray}\label{eq:gammadef}
    \gamma(u) = \frac{\beta u\ct A\zz Au}{u\ct S u} \quad \mbox{for all}\,\,
    u\in\bC^n \,\,\mbox{such that} \quad Su\neq 0,
    \end{eqnarray}
    where $u^*$ is the complex conjugate of $u$.
    Recalling that $S =H+\beta A^\top A$, we know
    \begin{eqnarray}\label{eq:gammarange}
    0\leq \gamma(u)\leq 1 \quad \mbox{for all}\,\, u\in\bC^n \,\,\mbox{such that} \quad Su\neq 0.
    \end{eqnarray}
    Similarly, we define
    \begin{eqnarray}\label{eq:kappadef}
    \kappa(u) = \frac{u\ct Q\inv u}{u\ct Su}\quad \mbox{for all}\,\,
    u\in\bC^n \,\,\mbox{such that} \quad Su\neq 0.
    \end{eqnarray}
    Note that $\mbox{eig}(QS)<{4\over 3}$ by  Lemma \ref{lemma:eigQS}. Thus,
    we know that
    $
    {4\over 3}Q^{-1} -S \succeq 0,
    $
    and therefore
    \begin{eqnarray}\label{eq:kapparange}
    0 < \kappa(u)\inv < \frac{4}{3} \quad \mbox{for all}\,\, u\in\bC^n \,\,\mbox{such that} \quad Su\neq 0.
    \end{eqnarray}}
  	Note that $M$ can be factorized as
    \begin{eqnarray}\label{eq:Mfact}
    M
    =
    \left[
    \begin{array}{cc}
    I & 0\\
    -\beta A & I
    \end{array}
    \right]
    \left[
    \begin{array}{cc}
    I-QS & QA\zz\\
    0 & I
    \end{array}
    \right].
    \end{eqnarray}
    Switching the order of the products, we obtain a new matrix
    \begin{eqnarray}
    M':= \left[
    \begin{array}{cc}
    I-QS & QA\zz\\
    0 & I
    \end{array}
    \right]\left[
    \begin{array}{cc}
    I & 0\\
    -\beta A & I
    \end{array}
    \right]
    = \left[
    \begin{array}{cc}
    I-QS - \beta QA\zz A & QA\zz\\
    -\beta A & I
    \end{array}
    \right].
    \end{eqnarray}
    Note that $\mbox{eig}(M)=\mbox{eig}(M')$. Thus, it suffices to show
      either $\rho(M')<1$ or $1$ is the eigenvalue of $M'$.

    Let  $\left(\lambda,
    \left[
    \begin{array}{c}
    v_1 \\
    v_2
    \end{array}
    \right]
    \right)$ be an eigenpair of $M'$, namely,
    \begin{eqnarray*}
        \left[
        \begin{array}{cc}
            I-QS - \beta QA\zz A & QA\zz\\
            -\beta A & I
        \end{array}
        \right] \left[
        \begin{array}{c}
            v_1 \\
            v_2
        \end{array}
        \right] = \lambda \left[
        \begin{array}{c}
            v_1 \\
            v_2
        \end{array}
        \right],
    \end{eqnarray*}
    which implies
    \begin{eqnarray}\label{eq:eig1}
    (I-QS -
    \beta QA\zz A)v_1 + QA\zz v_2 &=& \lambda v_1;\\
    \label{eq:eig2}
    -\beta Av_1 + v_2 &=& \lambda v_2.
    \end{eqnarray}
    Equality \eqref{eq:eig2} gives
    \begin{eqnarray}\label{eq:eig3}
    (1-\lambda) v_2 = \beta Av_1.
    \end{eqnarray}
    Suppose $\lambda\neq 1$.
    Hence, it holds that
    \begin{eqnarray*}
        v_2=\frac{\beta}{1-\lambda}Av_1.
    \end{eqnarray*}
    Clearly, this relation implies that $v_1\neq 0$.
    Substituting the above relation into
    \eqref{eq:eig1}, we have
    \begin{eqnarray*}
        QSv_1 = (1-\lambda)v_1 + \frac{\lambda\beta}{1-\lambda}QA\zz Av_1.
    \end{eqnarray*}
   Using the nonsingularity of $Q$, the above equality can be written as
    \begin{eqnarray*}
        Sv_1 = (1-\lambda)Q\inv v_1 + \frac{\lambda\beta}{1-\lambda}A\zz Av_1.
    \end{eqnarray*}
    Multiplying both sides of the above equality by $v_1\ct$, we arrive at
    \begin{eqnarray}\label{eq:eig5}
    v_1\ct Sv_1 = (1-\lambda)v_1\ct Q\inv v_1 +
    \frac{\lambda\beta}{1-\lambda}v_1\ct A\zz Av_1,
    \end{eqnarray}
    We claim that $v_1^*Sv_1\neq 0$. Otherwise,  $v_1^*A^\top Av_1=0$
    and therefore $\lambda=1$ from the inequality $v_1^*Q^{-1}v_1>0$ and \eqref{eq:eig5}.
    This contradicts our assumption that $\lambda\neq 1$.
    %From the positive definiteness of $S$ and the fact that $v_1$ is nonzero,
    %the left hand side of \eqref{eq:eig5} is positive.
    Multiplying both sides of \eqref{eq:eig5} by $(v_1\ct Sv_1)\inv$ and
    substituting the definitions \eqref{eq:gammadef} and \eqref{eq:kappadef} into
    the above relation, we  obtain the following key equality with respect to $\lambda$
    \begin{equation*}
        1 = (1-\lambda)\kappa(v_1) + \frac{\lambda}{1-\lambda}\gamma(v_1),
    \end{equation*}
    which can be further reformulated as
    \red{\begin{equation*}
        \kappa(v_1)\lambda^2 - (2\kappa(v_1)-\gamma(v_1) - 1)\lambda  + \kappa(v_1) -1 =0.
    \end{equation*}}
    Because $\kappa(v_1)$ is positive, we have
    \begin{eqnarray}\label{eq:equation}
    \lambda^2 + \left(\kappa(v_1)\inv(\gamma(v_1) + 1)-2\right) \lambda
    + \left(1 -\kappa(v_1)\inv\right) =0.
    \end{eqnarray}
The discriminant of the quadratic equation in \eqref{eq:equation} is
    \begin{eqnarray}\label{eq:discri}
    \Delta &=& \left(\kappa(v_1)\inv(\gamma(v_1) + 1)-2\right)^2
    - 4\left(1 -\kappa(v_1)\inv\right)\nonumber\\
    &=& \kappa(v_1)\inv \left(\kappa(v_1)\inv (\gamma(v_1)+1)^2 -4\gamma(v_1)\right).
    \end{eqnarray}
   Note that
    \begin{eqnarray*}
        0\leq   \frac{4\gamma(v_1)}{(\gamma(v_1)+1)^2} \leq 1
    \end{eqnarray*}
    holds as a result of \eqref{eq:gammarange}.
    Recalling \eqref{eq:kapparange},
    we consider the following two cases.
    \begin{itemize}
        \item[Case 1:] $0<\kappa(v_1)\inv<\frac{4\gamma(v_1)}{(\gamma(v_1)+1)^2}$.
        This means the discriminant $\Delta<0$, and the two solutions  of \eqref{eq:equation} satisfy
        \[
        |\lambda_{1,2}|=\sqrt{\lambda_1*\lambda_2}= \sqrt{1-\kappa(v_1)^{-1}}<1.
        \]
        \item[Case 2:] $\frac{4\gamma(v_1)}{(\gamma(v_1)+1)^2}\leq \kappa(v_1)\inv <\frac{4}{3}$.
        This means the discriminant $\Delta\geq 0$, and the two solutions
        are real. Let
        \[
        f(\lambda) := \lambda^2 + \left(\kappa(v_1)\inv(\gamma(v_1) + 1)-2\right) \lambda
        + \left(1 -\kappa(v_1)\inv\right).
        \]
        By \eqref{eq:gammarange} and \eqref{eq:kapparange}, we know that
        \[
        \left\{
        \begin{array}{l}
        f(1) = {\gamma(v_1) \over \kappa(v_1) }\geq 0, \\[0.2cm]
        f(-1) =  4-{{\gamma(v_1)+2}\over \kappa(v_1)}>0,\\[0.2cm]
        \lambda_1 +\lambda_2 = 2- {\gamma(v_1)+1\over \kappa(v_1)}\in (-2,\,2),
        \end{array}
        \right.
        \]
        which together with $\lambda \neq 1$, establishes that $|\lambda|<1$.
    \end{itemize}
    Thus, it can be concluded that either  $\lambda =1$ or
    $|\lambda|<1$ holds. \quad
\end{proof}

We now consider the case where $M$ has an eigenvalue equal to 1 and show that it has  a complete
set of eigenvectors.

\begin{lemma}\label{lemma:diag}
    Suppose that Assumption {\em \ref{asmp:4}} holds,
    and $M\in\R^{(m+d)\times (m+d)}$ is a matrix defined by \eqref{eq:Mdef}.
    Suppose that $1$ is an eigenvalue of $M$,
    then the algebraic multiplicity of $1$ for $M$ equals its geometric multiplicity. Namely,
    the eigenvalue $1$ has a complete set of eigenvectors.
\end{lemma}
\begin{proof}
 By direct computation, it holds that
    \begin{eqnarray}
    {\rm det}( \lambda I -M ) &= &{\rm det}\left[
    \begin{array}{cc}
    (\lambda-1)I + QS & -QA\zz\\
    \beta A-\beta AQS & (\lambda-1)I + \beta AQA\zz
    \end{array}
    \right] \nonumber \\
    & =&{\rm det} \left[
    \begin{array}{cc}
    (\lambda-1)I + QS & -QA\zz\\
    \lambda\beta A & (\lambda-1)I
    \end{array}
    \right] \nonumber\\
    &= &
    {\rm det} \left[
    \begin{array}{cc}
    \displaystyle (\lambda-1)I + QS  +{\lambda\beta\over \lambda-1}QA^\top A& -QA\zz \nonumber\\[0.3cm]
    0 & (\lambda-1)I
    \end{array}
    \right]\\[0.3cm]
    & =& (\lambda-1)^{m-d}\, {\rm det} \left[(\lambda-1)^2I  + (2\lambda-1)\beta QA^\top A + (\lambda-1)Q H \right] \nonumber \\[0.3cm]
    & =  & (\lambda-1)^{m-d}\,{\rm det} \left[(\lambda-1)^2I  + (2\lambda-1)\beta Q^{1/2}A^\top A Q^{1/2}+ (\lambda-1) Q^{1/2}H Q^{1/2}  \right]\nonumber.
    \end{eqnarray}
    This, together with Lemma \ref{lemma:alg}, shows that the algebraic multiplicity of $1$ for $M$ equals
    \begin{eqnarray}\label{eq:AM1}
    && m-d +2d- {\rm Rank} (Q^{1/2}\beta A^\top A Q^{1/2}) - {\rm Rank} (Q^{1/2}(\beta A^\top A +H) Q^{1/2})\nonumber \\[0.1cm]
    &  & \;= m+ d- {\rm Rank} (\beta A^\top A) - {\rm Rank} (\beta A^\top A +H),
    \end{eqnarray}
    where the equality follows from $Q\succ 0$ by Lemma \ref{lemma:eigQS}.
    In addition, the geometric multiplicity of  $1$ for $M$ is identical to the following quantity:
    \red{\begin{eqnarray}\label{eq:GM1}
    & & m+d -{\rm Rank} (I-M) \nonumber \\
    &  &\; = m+d -{\rm Rank} \left[
    \begin{array}{cc}
    QS & -QA\zz\\
    \beta A-\beta AQS &  \beta AQA\zz
    \end{array}
    \right] \nonumber\\
    &  &\; =  m+d -{\rm Rank} \left[
    \begin{array}{cc}
    QS & -QA\zz\\
    \beta A &  0
    \end{array}
    \right] \\
    &  &\; =  m+d -{\rm Rank} \left[
    \begin{array}{cc}
    S & -A\zz\\
    \beta A &  0
    \end{array}
    \right],\nonumber
    \end{eqnarray}}
    where the second equality follows from the rank invariant property under elementary transformation, and the final equality holds because $Q\succ 0$ by Lemma \ref{lemma:eigQS}. Combining \eqref{eq:AM1}, \eqref{eq:GM1},
    Lemma \ref{lemma:geo}, and the definition of $S$, we derive the desired
    conclusion.\quad
\end{proof}

\subsection{Expected Convergence}\label{sec:conv-3}
Step (4) can be formulated as the following theorem.
\begin{theorem}\label{thm:conv}
Assume Assumption {\em \ref{asmp:4}} holds. Suppose  RPADMM \eqref{eq:rpadmm} is employed to solve the nonseparable quadratic programming \eqref{eq:quadratic}. Then, the expected iterative sequence converges to some KKT point of
   \eqref{eq:quadratic}.
   \end{theorem}

\begin{proof}
Let $(\bar x,\bar \mu)$ be a KKT point of \eqref{eq:quadratic}, i.e.,
\begin{eqnarray}\label{eq:optimality1}
\left[
\begin{array}{cc}
H & -A\zz \\
\beta A & 0
\end{array}
\right]
\left[
\begin{array}{c}
\bar x \\
\bar \mu
\end{array}
\right]
=
\left[
\begin{array}{c}
-g \\
\beta b
\end{array}
\right].
\end{eqnarray}
Denote $(x^k,\,\mu^k)$ by the $kth$ iterate of the algorithm. It follows from \red{\eqref{eq:RPADMM} and} \eqref{eq:optimality1} that
    \[
    E_{\sigma}[x^{k+1}-\bar x;\mu^{k+1}-\bar \mu] = M E_{\sigma} [x^{k}-\bar x;\mu^{k}-\bar \mu].
    \]
    By Lemma \ref{lemma:spectral}, we know that  $\rho(M)\leq 1$. We proceed with the proof by considering the following two cases.
    \begin{itemize}
        \item[Case 1:] $\rho(M)<1$. It holds that
    $
    E_\sigma x^k \rightarrow \bar x
    $
    and
    $
    E_\sigma \mu^k \rightarrow \bar \mu
    $
    as $k\rightarrow \infty$.
    Theorem \ref{thm:conv} is valid.
        \item[Case 2:] $\rho(M)=1$. By Lemmas \ref{lemma:spectral} and \ref{lemma:diag}, we know that all eigenvalues of $M$ with modulus $1$ must be $1$, which has a complete set of eigenvectors.
   As a result, $M$ admits the following Jordan decomposition:
    \[
    M = P^{-1}\left[
    \begin{array}{cccccc}
    1& & & & & \\
    &\ddots& & & & \\
    && 1 & & &\\
    &&   & \rho_1 &* &\\
    &&  &             &\ddots &*\\
    &&  &            &            & \rho_t
    \end{array}
    \right] P,
    \]
    where $P$ is a nonsingular matrix and $|\rho_i|<1$ for all $i=1,\ldots,t$. It is easily verified that
    \[
    M^k \rightarrow  P^{-1}
    \left[
    \begin{array}{cccccc}
    1& & & & & \\
    &\ddots& & & & \\
    && 1 & & &\\
    &&   & 0 & &\\
    &&  &             &\ddots &\\
    &&  &            &            & 0
    \end{array}\right]P
    \]
    as $k\rightarrow \infty$, and therefore the sequence $\{E[x^{k+1}-\bar x;\mu^{k+1}-\bar \mu]\}$ converges to an eigenvector of $M$
    associated with the eigenvalue $1$, say $[x^0;\mu^0]$. Then
    \[
    (I-M) [x^0;\mu^0] =0,
    \]
    which, after some manipulation, shows that
    \begin{eqnarray}
    \left[
    \begin{array}{cc}
    H & -A\zz \\
    \beta A & 0
    \end{array}
    \right]
    \left[
    \begin{array}{c}
    x^0 \\
    \mu^0
    \end{array}
    \right]
    = 0.
    \end{eqnarray}
    Therefore, $Ex^k \rightarrow \bar x + x^0$ and $E\mu^k \rightarrow \bar \mu  + \mu^0$ with
    \begin{eqnarray}
    \left[
    \begin{array}{cc}
    H & -A\zz \\
    \beta A & 0
    \end{array}
    \right]
    \left[
    \begin{array}{c}
    \bar x +x^0 \\
    \bar \mu +\mu^0
    \end{array}
    \right]
    =
    \left[
    \begin{array}{c}
    -g \\
    \beta b
    \end{array}
    \right].
    \end{eqnarray}
    This means that $(\bar x+x^0, \bar \mu+\mu^0)$ is a KKT point of \eqref{eq:quadratic}.
    \end{itemize}
   This completes the proof. \quad
\end{proof}

One byproduct of Theorem \ref{thm:conv} is the expected convergence result for RPBCD  when applied to convex quadratic optimization. To the best of our knowledge, this is the first expected iterate convergence result of RPBCD.
\begin{corollary}\label{Conver-Alg2-e}
Assume $H_{ii}\succ 0$ for $i=1,2,\ldots,n$. If RPBCD is used to solve the unconstrained quadratic programming problem
\begin{equation}\label{prob:quun}
\displaystyle \min_{x\in \R^d}\,\,  {1\over 2}x^\top Hx +g^\top x,  \\[5pt]
\end{equation}
then the expected iterative sequence converges to an optimal solution of
\eqref{prob:quun}.
\end{corollary}

\subsection{Convergence Rate   Comparison to Cyclic BCD}
There is a common perception that RPBCD dominates cyclic BCD in terms of performance (see  \cite{Steve2015}, for example). In this subsection, we theoretically show that this is not generally true. Consider the quadratic programming problem
\eqref{prob:quun}, where $x$ is split into two blocks
$(x_1,\,x_2)$ with $x_1\in \R^{d_1}$ and $x_2\in \R^{d_2}$,
and $d=d_1+d_2$.
Accordingly, we denote
\[
H=\left[\begin{array}{cc}
H_{11} & H_{12}\\ [0.1cm]
H_{12}\zz & H_{22}\\
\end{array}\right].
\]
By applying different  minimizaing orders to the variables, the cyclic BCD (Gauss-Seidel method) has the following two iterative
schemes:
\begin{eqnarray*}
x^{k+1} = M_1 x^k - \left[\begin{array}{cc}
H_{11} & 0\\[0.1cm]
H_{12}\zz & H_{22}\\
\end{array}\right]\inv b
\end{eqnarray*}
and
\begin{eqnarray*}
x^{k+1} = M_2 x^k - \left[\begin{array}{cc}
H_{11} & H_{12}\\[0.1cm]
0 & H_{22}\\
\end{array}\right]\inv b,
\end{eqnarray*}
where
\begin{eqnarray}\label{eq:M1M2}
M_1=\left[\begin{array}{cc}
H_{11} & 0\\ [0.1cm]
H_{12}\zz & H_{22}\\
\end{array}\right]\inv
\left[\begin{array}{cc}
0 & -H_{12}\\
0 & 0\\
\end{array}\right]\,\,\,\,{\rm and}\,\,\,\,
M_2=\left[\begin{array}{cc}
H_{11} & H_{12}\\[0.1cm]
0 & H_{22}\\
\end{array}\right]\inv
\left[\begin{array}{cc}
0 & 0\\ [0.1cm]
-H_{12}\zz & 0\\
\end{array}\right].
\end{eqnarray}
The asymptotic convergence rates of these two iterative schemes are
$\rho(M_1)$ and $\rho(M_2)$, respectively. In this case, the expected asymptotic convergence rate of RPBCD is $\rho((M_1+M_2)/2)$.
The following proposition reveals the relationship between these rates.

\begin{proposition}
Suppose %\remove{$H\succeq 0$,} 
$H_{11}\succ 0$ and $H_{22}\succ 0$.
Let $M_1$ and $M_2$ be defined by \eqref{eq:M1M2},
and $M_3=(M_1+M_2)/2$. Then, it holds that
    $$\rho(M_1) = \rho(M_2) \leq \rho(M_3).$$
\end{proposition}
\begin{proof} Without loss of generality, we  need only consider the situation where $H_{ii}=I_{d_i}$
    for $i=1,2$ and $d_1\geq d_2$ because the similarity transformation
    $M\mapsto  P M P\inv$
    does not
    change the spectrum of $M$, where $P=\left[\begin{array}{cc}
    H_{11}^{\frac{1}{2}} & 0\\
    0 & H_{22}^{\frac{1}{2}}\\
    \end{array}\right]$.
    In this case, a simple calculation yields
    \begin{equation}\label{eq:iterM12}
    M_1 =  \left[\begin{array}{cc}
    0 & -H_{12}\\ [0.1cm]
    0 & H_{12}\zz H_{12}\\
    \end{array}\right]\qquad {\rm and}
    \qquad
    M_2 =  \left[\begin{array}{cc}
    H_{12}H_{12}\zz & 0\\[0.1cm]
    -H_{12}\zz & 0\\
    \end{array}\right].
    \end{equation}
    Let $\sigma_1\geq \sigma_2\geq \ldots \geq\sigma_{d_2}$ be the eigenvalues of $H_{12}^\top H_{12}$.
    Recall that $H \succeq 0$ and $H_{ii} = I_{d_i}$ for $i=1,2$. Then, we have that
    $\sigma_i \in [0, 1]$, $i=1, \ldots, d_2$, and obtain from
    \eqref{eq:iterM12} that
    \begin{equation*}\label{eq:spect12}
    \rho(M_1)=\rho(M_2) = \sigma_1.
    \end{equation*}

    Clearly,
    $$M_3 =
    \frac{1}{2}\left[\begin{array}{cc}
    H_{12}H_{12}\zz & -H_{12}\\[0.1cm]
    -H_{12}\zz & H_{12}\zz H_{12}\\
    \end{array}\right].
    $$
    By direct computation, it holds that
    \begin{eqnarray*}
    {\rm det}( \lambda I -M_3 ) &= &\left(1\over 2\right)^d{\rm det}\left[
    \begin{array}{cc}
     2\lambda I- H_{12}H_{12}^\top  & H_{12}\\[0.2cm]
     H_{12}^\top & 2\lambda I - H_{12}^\top H_{12}
    \end{array}
    \right] \nonumber \\[0.3cm]
    & =&
    \left({1\over 2}\right)^d(2\lambda)^{d_1-d_2}
    {\rm det} \left[4\lambda^2 I -(4\lambda+1)H_{12}^\top H_{12}+(H_{12}^\top H_{12})^2\right]
    \nonumber\\[0.2cm]
    &=&\left({1\over 2}\right)^d(2\lambda)^{d_1-d_2}\prod_{i=1}^{d_2}(4\lambda^2-(4\lambda+1)\sigma_i +\sigma_i^2)
    \end{eqnarray*}
    and so the eigenvelues of $M_3$ are $0$ (multiplicty $=d_1-d_2$) and ${\sigma_i \pm \sqrt{\sigma_i}\over 2}$
    for $i=1,2,\ldots,d_2$. Because $\sigma_1\in [0,1]$, we  have that
    \[
    \rho(M_3) ={\sigma_1 +\sqrt{\sigma_1}\over 2}\geq \sigma_1.
    \]
   This completes the proof. \quad \end{proof}

Therefore, although random permutation does indeed make multi-block ADMM and BCD more robust, especially for  ``bad" or diverging problems, cyclic ADMM or BCD may still perform well, or even better, for solving ``nice" problems.

\section{Concluding Remarks}

In this paper, we have demonstrated the point-wise or iterate convergence of the classical
$2$-block  ADMM for solving convex optimization problems with coupled quadratic objective functions under a mild assumption. This assumption becomes necessary and sufficient for the global convergence of the ADMM when the objective is a quadratic function. This result partially answers, in the affirmative, the open question arising in  \cite{HLR2014} on the convergence of ADMM for nonseparable optimization problems. We also derived the expected convergence of RPADMM in solving linearly constrained coupled quadratic optimization problems. This is a non-trivial extension of the convergence analysis given  in
\cite{SLY2015}, which is only applicable to nonsingular linear systems. When the linear constraint is absent, the proximal ADMM and RPADMM reduce to the cyclic proximal BCD and RPBCD. Thus, this study has provided new convergence results for BCD-type methods. In particular, we have established the first iterate convergence result for 2-block cyclic \red{proximal} BCD without assuming the boundedness of the iterates and the expected iterate convergence of RPBCD for multi-block convex quadratic optimization. We also theoretically demonstrated that RPBCD does not necessarily dominate cyclic BCD. Although the results for RPADMM and RPBCD are restricted to quadratic minimization models, they provide some interesting insights on the use of these methods: 1) random permutation makes multi-block ADMM and BCD more robust for multi-block convex minimization problems; 2) cyclic BCD may outperform RPBCD for ``nice'' problems, and therefore RPBCD should be applied with caution when solving general multi-block convex optimization problems.

  Two challenging open questions concern the extension of our convergence results for RPADMM and RPBCD to more general convex optimization problems, and an exploration of the global convergence rate of RPADMM and RPBCD. In particular, it would be interesting to know which problems are better suited to RPADMM or RPBCD.

% the convergence results of  the ADMM and BCD (the RPADMM and EPOCHS)
% are restricted to the nonsepearable convex minimization model \eqref{eq:obj1} and its unconstrained
% version, respectively, so that it should be used in caution.

\bigskip

\noindent {\bf Acknowledgements.} Caihua Chen was supported by the
Natural Science Foundation of Jiangsu Province [Grant No.
BK20130550] and the National Natural Science Foundation of China
[Grant No. 11371192]. Min Li was supported by the National Natural
Science Foundation of China [Grant No. 11001053, 71390335], Program
for New Century Excellent Talents in University [Grant No.
NCET-12-0111], and Qing Lan Project. Xin Liu was supported by the
National Natural Science Foundation of China [Grant No. 11471325, 91530204,
11622112, 11331012, 11461161005], the National Center for Mathematics and Interdisciplinary Sciences,
CAS, and Key Research Program of Frontier Sciences, CAS.
Yinyu Ye was supported by the AFOSR Grant [Grant No. FA9550-12-1-0396].

The authors would like to thank Dr. Ji Liu from University of
Rochester and Dr. Ruoyu Sun from Stanford University
for the helpful discussions on the block coordinate descent
method.

\bigskip

\bigskip

\newpage

\noindent{\Large \bf Appendix.}

\bigskip

\noindent{\bf Appendix A.} \blue{
The proof of Lemma \ref{lemma:eigQS} is similar to, but not exactly the same as,
that of \cite[Lemma 2]{SLY2015}. Since $S$ is allowed to be singular here,
we need also show the positive definiteness of $Q$ by mathematical induction.
For completeness, we will provide a concise proof here.
Interested readers are referred to \cite{SLY2015} for the motivation and other details
of this proof.

{\bf Proof of Lemma \ref{lemma:eigQS}.}
This lemma reveals a linear algebra property, and is essentially not related with $H$, $A$ and $\beta$
if we define $L_\sigma$ directly by $S$. For brevity, we
restate the main assertion to be proved as following:
\begin{eqnarray}\label{eq:eig}
\mathrm{eig}(QS)\subset \left[0,\frac{4}{3}\right),
\end{eqnarray}
where $S\in\R^{d\times d}$ is positive semidefinite, $S_{ii}\in\R^{d_i\times d_i}$ ($i=1,...,n$)
is positive definite,
\begin{eqnarray}\label{eq:defineL}
(L_{\sigma})_{\sigma(i),\sigma(j)}:=\left\{
\begin{array}{ll}
S_{\sigma(i)\sigma(j)}, & \hbox{if} \;\; 1\leq j \leq i \leq n,\\
0, & \hbox{otherwise},
\end{array}
\right.  \qquad Q:=\frac{1}{n!}\sum\limits_{\sigma\in\Gamma} L\inv_\sigma,
\end{eqnarray}
and $\Gamma$ is a set consisting of all permutations of $(1,...,n)$.

Without loss of generality, we assume $S_{ii}=I_{d_i}$ ($i=1,...,n$). Otherwise,
we denote $$D:=\mathrm{Diag}\left(S_{11}^{-\frac{1}{2}},...,
S_{nn}^{-\frac{1}{2}}\right).$$ It is easy to verify that
$\tilde{Q} = D\inv QD\inv $, if $\tilde{S} =DSD$, and $\tilde{L}_{\sigma}$ and $\tilde{Q}$
are defined by \eqref{eq:defineL} with $\tilde{S}$. It holds that
$$\mathrm{eig}(\tilde{Q}\tilde{S})=\mathrm{eig}(D\inv QD\inv DSD)=
\mathrm{eig}(D\inv QSD)=\mathrm{eig}(QS),$$
and $\tilde{S}_{ii}=I_{d_i}$ ($i=1,...,n$). Due to the positive semi-definiteness of $S$,
and by a slight abuse of the notation $A$,
there exists $A\in\R^{d\times d}$ satisfying $S=A\zz A$. Let $A_i\in\R^{d\times d_i}$ ($i=1,...,n$)
be the column blocks of $A$, and it is clear that $S_{ij} = A_i\zz A_j$ for all $1\leq i,j\leq n$.
In addition, it also holds that $\mathrm{eig}(QS)=\mathrm{eig}(AQA\zz)$.

For the brevity of notation, we define the block permutation matrix $P_k$
as following:
\begin{eqnarray}\label{eq:defP}
(P_k)_{ij}:=\left\{
\begin{array}{ll}
I_{d_i}, & \mbox{if\,} 1\leq i=j\leq k-1;\\
I_{d_i}, & \mbox{if\,} k+1\leq i=j+1\leq n ;\\
I_{d_i}, & \mbox{if\,} i=k,\,j=n;\\
0_{d_i\times d_j}, & \mbox{if\,} 1\leq j\leq k-1,\, i\neq j;\\
0_{d_i\times d_{j+1}}, &\mbox{if\,} k\leq j \leq n-1,\, i\neq j+1;\\
0_{d_i\times d_k}, &\mbox{otherwise.}
\end{array}
\right.
\end{eqnarray}
It can be easily verified that $P_k\zz = P_k\inv$, and $P_n=I_d$.
For $k\in(1,...,n)$, we define
$\Gamma_k:=\{\sigma'\mid \sigma' \mbox{\,is\,a\,permutation\,of\,} (1,...,k-1,k+1,...,n)\}$.
For any $\sigma'\in \Gamma_k$, we define $L_{\sigma'}\in\R^{(d-d_k)\times (d-d_k)}$
as the following
\begin{eqnarray}\label{eq:defL'}
(L_{\sigma'})_{\sigma'(i),\sigma'(j)}:=\left\{
\begin{array}{ll}
S_{\sigma'(i)\sigma'(j)}, & \hbox{if} \;\; 1\leq j \leq i \leq n-1,\\
0, & \hbox{otherwise}.
\end{array}
\right.
\end{eqnarray}
We define $\hat{Q}_k\in\R^{(n-d_k)\times (n-d_k)}$ by
\begin{eqnarray}\label{eq:Qk}
\hat{Q}_k := \frac{1}{|\Gamma_k|}\sum\limits_{\sigma'\in\Gamma_k}L_{\sigma'}\inv,\qquad
k=1,...,n,
\end{eqnarray}
and $W_k$ as the $k$-th block-column of $S$ excluding the block $S_{kk}$,
i.e. $W_k =[S_{k1},...,S_{kn}]\zz$. Moreover, let
$\hat{A}_k:=[A_1,...,A_{k-1},A_{k+1},...,A_n]$, we have
$AP_k = [\hat{A}_k,A_k]$.

Now we use mathematical induction to prove this lemma.
Firstly, the assertion \eqref{eq:eig} and $Q\succ 0$ hold when $n=1$, as $QS=I$ in this case.
Next, we will prove the lemma for any $n\geq 2$ given that the assertion
\eqref{eq:eig} and $Q\succ 0$ hold for $n-1$.

A key step of the proof is to reveal the following relationship.
\begin{eqnarray}\label{eq:key1}
Q=\frac{1}{n}\sum\limits_{k=1}^n P_k Q_k P_k\zz,
\end{eqnarray}
where
\begin{eqnarray}\label{eq:defQk}
Q_k:=\left[
\begin{array}{cc}
\hat{Q}_k & -\frac{1}{2}\hat{Q}_k W_k\\
-\frac{1}{2}W_k\zz \hat{Q}_k & I_{d_k}
\end{array}
\right],
\end{eqnarray}
in which $\hat{Q}^k$ is defined by \eqref{eq:Qk}.
The proof of \eqref{eq:key1} will be provided later.

It directly follows from \eqref{eq:key1} that
$AQA\zz = \frac{1}{n}\sum\limits_{k=1}^n AP_k Q_kP_k\zz A\zz$.
Consequently,
\begin{eqnarray}\label{eq:subineq}
	\frac{1}{n}\sum\limits_{k=1}^n \lambda_{\min}(AP_kQ_kP_k\zz A\zz)
	\leq \lambda_{\min}(AQA\zz)\leq \lambda_{\max}(AQA\zz)
	\leq \frac{1}{n}\sum\limits_{k=1}^n \lambda_{\max}(AP_kQ_kP_k\zz A\zz).
\end{eqnarray}

We will show, in the end of this proof, the fact that
\begin{eqnarray}\label{eq:key2}
\mathrm{eig}(AQ_nA\zz)\subset \left[0,\frac{4}{3}\right)
\end{eqnarray}
if it holds that
\begin{eqnarray}\label{eq:asmpn-1}
	\mathrm{eig}(\hat{Q}_n \hat{A}_n\zz \hat{A}_n)\subset \left[0,\frac{4}{3}\right).
\end{eqnarray}
In fact, \eqref{eq:asmpn-1} holds directly by the induction assumption.
Together with the similarity among the blocks, the relationship \eqref{eq:key2}
implies
\begin{eqnarray}\label{eq:asmpn-2}
\mathrm{eig}(AP_kQ_kP_k\zz A\zz)\subset \left[0,\frac{4}{3}\right),\qquad \mbox{for\,all\,}
k=1,...,n.
\end{eqnarray}
Substitute \eqref{eq:asmpn-2} into \eqref{eq:subineq},
we prove the assertion \eqref{eq:eig} for $n$, and hence complete the proof of
Lemma \ref{lemma:eigQS}.

Our remaining task is to prove the relationships \eqref{eq:key1} and \eqref{eq:key2}.
We will achieve this goal by the following two steps.

\noindent {\bf Step 1.} Let $\sigma'\in\Gamma_k$, we can partition $L_{\sigma'}$
as following
\begin{eqnarray}\label{eq:63}
L_{\sigma'} =\left[
\begin{array}{cc}
Z_{11} & Z_{12}\\
Z_{21} & Z_{22}
\end{array}
\right].
\end{eqnarray}
Here the sizes of $Z_{11}$ and $Z_{22}$
are $(d_1+\cdots + d_{k-1})\times (d_1+\cdots + d_{k-1})$
and $(d_{k+1}+\cdots + d_{n})\times (d_{k+1}+\cdots + d_{n})$,
respectively. The sizes of $Z_{12}$ and $Z_{21}$ can be determined accordingly.
We denote
\begin{eqnarray*}
	U_k=(A_1,...,A_{k-1}),\qquad V_k = (A_{k+1},...,A_n),
\end{eqnarray*}
which implies
\begin{eqnarray}\label{eq:64}
W_k = [U_k,V_k]\zz A_k = \mat{c}{U_k\zz A_k\\ V_k\zz A_k}.
\end{eqnarray}
It is then easy to verify that
\eqnonum{
L_{(\sigma',k)}
=\mat{ccc}{
	Z_{11} & U_k\zz A_k & Z_{12}\\
	0 & I_{d_k} & 0\\
	Z_{21} & V_k\zz A_k & Z_{22}
}.}
Left and right multiplying both sides of the above relationship
by  $P_k\zz$ and $P_k$, respectively, we obtain
\eqnum{eq:65}{
P_k\zz L_{(\sigma',k)} P_k
= P_k\zz \mat{ccc}{
Z_{11} & Z_{12} & U_k\zz A_k \\
0  & 0 & I_{d_k}\\
Z_{21}  & Z_{22} & V_k\zz A_k
}
=\mat{ccc}{
Z_{11} & Z_{12} & U_k\zz A_k \\
Z_{21}  & Z_{22} & V_k\zz A_k\\
0  & 0 & I_{d_k}\\
}
=\mat{cc}{
L_{\sigma'} & W_k\\
0 & I_{d_k}
}.
}
Taking the inverse of both sides of \eqref{eq:65}, we obtain
\eqnum{eq:62}{
P_k\zz L_{(\sigma',k)}\inv P_k
=\mat{cc}{
	L_{\sigma'}\inv & -L_{\sigma'}\inv W_k\\
	0 & I_{d_k}
}.
}

Summing up \eqref{eq:62} for all $\sigma'\in\Gamma_k$
and dividing by $|\Gamma_k|$, we get
\begin{eqnarray}\label{eq:66}
\frac{1}{|\Gamma_k|}
\sum\limits_{\sigma'\in \Gamma_k}
P_k\zz L_{(\sigma',k)}\inv P_k
=\mat{cc}{
\frac{1}{|\Gamma_k|}
\sum\limits_{\sigma'\in \Gamma_k}
\zz L_{(\sigma',k)}\inv &
-\frac{1}{|\Gamma_k|}
\sum\limits_{\sigma'\in \Gamma_k}
 L_{(\sigma',k)}\inv W_k\\
 0 & I_{d_k}
} = \mat{cc}{
\hat{Q}_k & -\hat{Q}_k W_k\\
0 & I_{d_k}
},
\end{eqnarray}
Here, the last equality follows from \eqref{eq:Qk}.
By the definition of $L_\sigma$, it is easy to verify that
$L_{\sigma}\zz = L_{\bar{\sigma}}$,
where $\bar{\sigma}$ is a ``reverse permutation" of $\sigma$
that satisfies $\bar{\sigma}(i)=\sigma(n+1-i)$ ($i=1,...,n$).
Thus we have $L_{(\sigma',k)}=L_{(k,\bar{\sigma}')}\zz$,
where $\bar{\sigma}'$ is a reverse permutation of $\sigma'$.
Summing over all $\sigma'$, we get
$$\sum\limits_{\sigma'\in \Gamma_k}L_{(\sigma',k)}\inv
= \sum\limits_{\sigma'\in \Gamma_k} L_{(k,\bar{\sigma}')}\nzz
= \sum\limits_{\sigma'\in \Gamma_k} L_{(k,\sigma')}\nzz,
$$
where the last equality follows from the fact that the summing over $\bar{\sigma}'$
is the same as summing over $\sigma'$. Thus, we have
\begin{eqnarray*}
	\frac{1}{|\Gamma_k|}
	\sum\limits_{\sigma'\in \Gamma_k}
	P_k\zz L_{(k,\sigma')}\inv P_k
	=\dkh{\frac{1}{|\Gamma_k|}
	\sum\limits_{\sigma'\in \Gamma_k}
	P_k\zz L_{(\sigma',k)}\inv P_k}\zz
	= \mat{cc}{
		\hat{Q}_k & 0\\
		-W_k\zz \hat{Q}_k  & I_{d_k}
	}.
\end{eqnarray*}
Here, the last equality uses the symmetry of $\hat{Q}_k$.
Combining the above relation, \eqref{eq:66} and the definition of
$Q_k$, we have
\begin{eqnarray}\label{eq:67}
\frac{1}{2|\Gamma_k|}P_k\zz
\sum\limits_{\sigma'\in \Gamma_k}
\dkh{
	L_{(k,\sigma')}\inv + L_{(\sigma',k)}\inv
}P_k =
\mat{cc}{
	\hat{Q}_k & -\frac{1}{2}\hat{Q}_k W_k\\
	-\frac{1}{2}W_k\zz \hat{Q}_k  & I_{d_k}
}=Q_k.
\end{eqnarray}

Using the definition of $P_k$ and the fact that $|\Gamma_k|=(n-1)!$,
we can rewrite \eqref{eq:67} as
\begin{eqnarray*}
	S_kQ_kS_k\zz = \frac{1}{2(n-1)!}\sum\limits_{\sigma'\in \Gamma_k}
	\dkh{
		L_{(k,\sigma')}\inv + L_{(\sigma',k)}\inv
	}.
\end{eqnarray*}
Summing up the above relation for $k=1,...,n$ and then dividing by $n$,
we immediately arrive at \eqref{eq:key1}.

\noindent {\bf Step 2.}  For simplicity, we use
$W$, $\hQ$ and $\hA$ to take the place $W_n$, $\hQ_n$
and $\hA_n$, respectively.

By the induction assumption, we have $\hQ\succ 0$, which implies
$\Theta:=W\zz \hQ W \succeq 0$. Recall that $S_{nn}=A_n\zz A_n = I_{d_n}$,
we have
\begin{eqnarray}\label{eq:69}
\rho(\Theta) =\max\limits_{v\in\R^{d_n},\,||v||=1}\,v\zz A_n\zz \hA\zz\hQ\hA A_nv
\leq \rho(\hA\hQ\hA) \max\limits_{v\in\R^{d_n},\,||v||=1}\,||A_nv||_2^2
<\frac{4}{3}||A_n||\fs =\frac{4}{3}.
\end{eqnarray}
Hence, we obtain
\begin{eqnarray}\label{eq:68}
0\preceq \Theta\prec \frac{4}{3}I_{d_n}.
\end{eqnarray}
Recall the definition \eqref{eq:defQk}, we have
\begin{eqnarray}\label{eq:70}
Q_n = \mat{cc}{
I_{d-d_n} & 0\\
-\frac{1}{2}W\zz & I_{d_n}
} \,\mat{cc}{
\hQ & 0\\
0 & I_{d_n} - \frac{1}{4}W\zz \hQ W
}\, \mat{cc}{
I & -\frac{1}{2}W\\
0 & I_{d_n}
}
=J\mat{cc}{
	\hQ & 0\\
	0 & C
}J\zz,
\end{eqnarray}
where $J:=\mat{cc}{
	I_{d-d_n} & 0\\
	-\frac{1}{2}W\zz & I_{d_n}
}$ and $C:=I_{d_n} - \frac{1}{4}W\zz \hQ W$.
Apparently, we have $C\succ 0$. Together with $\hQ\succ 0$,
it implies $Q_n\succ 0$. Thus, we directly obtain
$
\eig{AQ_nA\zz} \subset \left[0,\infty
\right)$.
It remains to show
\begin{eqnarray}\label{eq:73}
\rho(AQ_nA\zz) <\frac{4}{3}.
\end{eqnarray}
Denote $\hB:=\hA\zz \hA$, then we can write $S$ as
\begin{eqnarray*}
	S=A\zz A =\mat{cc}{
\hB & W\\
W\zz & I_{d_n}	
}.
\end{eqnarray*}
We can reformulate $\rho(AQ_nA\zz)$ as follows:
\begin{eqnarray}\label{eq:75}
\rho(AQ_nA\zz) = \rho\dkh{
AJ\mat{cc}{
	\hQ & 0\\
0 & C
}J\zz A\zz
}=\rho\dkh{
\mat{cc}{
	\hQ & 0\\
	0 & C
}J\zz A\zz A J
}.
\end{eqnarray}
It is easy to verify that
\begin{eqnarray*}
	J\zz A\zz A J
	= \mat{cc}{
		I_{d-d_n} & -\frac{1}{2}W\\
		0 & I_{d_n}
}\, \mat{cc}{
\hB & W\\
W\zz & I	
}\, \mat{cc}{
I_{d-d_n} & 0\\
-\frac{1}{2}W\zz & I_{d_n}
} = \mat{cc}{
\hB -\frac{3}{4}WW\zz & \frac{1}{2}W\\
\frac{1}{2}W\zz & I_{d_n}
}.
\end{eqnarray*}
Thus,
\begin{eqnarray}\label{eq:77}
Z:=\mat{cc}{
	\hQ & 0\\
	0 & C
}J\zz A\zz A J
=\mat{cc}{
\hQ\hB -\frac{3}{4}\hQ WW\zz & \frac{1}{2}\hQ W\\
\frac{1}{2}CW\zz & C
}.
\end{eqnarray}
According to \eqref{eq:75}, it suffices to prove $\rho(Z)<\frac{4}{3}$.
Suppose $\lambda$ is an arbitrary eigenvalue of $Z$, and $v\in \R^d$
is one of its associate eigenvector.
In the rest, we only need to show
\begin{eqnarray}\label{eq:78}
\lambda <\frac{4}{3}
\end{eqnarray}
holds. Then, using its arbitrariness,
we have $\rho(Z)<\frac{4}{3}$ which implies \eqref{eq:73}, and then \eqref{eq:key2}
holds.

Partition $v$ into $v=\mat{c}{v_1\\ v_0}$,
where $v_1\in \R^{d-d_n}$, $v_0\in \R^{d_n}$.
Then, $Zv=\lambda v$ implies that
\begin{eqnarray}\label{eq:79a}
\dkh{
\hQ\hB -\frac{3}{4}\hQ WW\zz
}v_1 + \frac{1}{2} \hQ W v_0 &=& \lambda v_1,\\
\label{eq:79b}
\frac{1}{2}CW\zz v_1 + Cv_0 &=& \lambda v_0.
\end{eqnarray}

If $\lambda I_{d_n}-C$ is singular, i.e. $\lambda$ is an eigenvalue of $C$.
By the definition of $C$ and \eqref{eq:68}, we have
$\frac{2}{3}I_{d_n}\prec C=I_{d_n}-\frac{1}{4}\Theta \preceq I_{d_n}$,
which implies that $\lambda \leq 1$, thus inequality \eqref{eq:78} holds.
In the following, we assume $\lambda I_{d_n}-C$ is nonsingular.
An immediate consequence is $v_1\neq 0$.

By \eqref{eq:79b}, we obtain $v_0=\frac{1}{2}(\lambda I_{d_n}-C)\inv CW\zz v_1$.
Substituting this explicit formula into \eqref{eq:79a}, we obtain
\begin{eqnarray}\label{eq:81}
\lambda v_1 = \dkh{
	\hQ\hB -\frac{3}{4}\hQ WW\zz
}v_1 + \frac{1}{4} \hQ W (\lambda I_{d_n}-C)\inv CW\zz v_1
= (\hQ\hB +\hQ W\Phi W\zz)v_1,
\end{eqnarray}
where
$
\Phi := -I_{d_n} +\lambda [(4\lambda -4) I_{d_n} +\Theta ]\inv$.
Since $\Theta$ is a symmetric matrix, $\Theta$ is also symmetric.

Suppose $\lambda_{\max}(\Phi)>0$,
the definition of $\Phi$ gives us
$$
\theta\in\eig{\Theta} \Leftrightarrow -1+\frac{\lambda}{(4\lambda -4)+\theta}
\in \eig{\Phi}.
$$
Together with $\lambda_{\max}(\Phi)>0$, there exists $\theta\in \eig{\Theta}$
such that $-1+\frac{\lambda}{(4\lambda -4)+\theta}$.
If $\lambda \leq 1$, \eqref{eq:78} already holds. Otherwise, $\lambda >1$,
which implies $1<\frac{\lambda }{(4\lambda-4)+\theta}\leq \frac{\lambda}{4\lambda -4}$,
and then \eqref{eq:78} holds.

Now we assume $\lambda_{\max}(\Phi)\leq 0$, i.e. $\Phi\preceq 0$.
By the induction, we have $\hl:=\rho(\hQ\hB)=
\rho(\hQ \hA\zz \hA) \subset\left[0,\frac{4}{3}\right)$.
Due to the positive definiteness of $\hQ$,
there exists nonsingular $U\in\R^{(d-d_n)\times (d-d_n)}$ such that
$\hQ =U\zz U$. Let $Y:=UW\Phi W\zz U\zz\in\R^{(d-d_n)\times (d-d_n)}$.

	We have
	$v\zz Yv
		=v\zz UW\Phi W\zz U\zz v
		=(W\zz U\zz v)\zz \Phi (W\zz U\zz v) \leq 0$
		holds for all $v\in\R^{d-d_n}$,
	where the last inequality follows from $\Phi\preceq 0$.
	Thus, $Y\preceq 0$. Pick up arbitrary $g$ satisfying
	$g>\rho(Y)$. Then, it holds that
	\begin{eqnarray}\label{eq:newadd1}
	\rho(g I_{d-d_n} +Y)\leq g.
	\end{eqnarray}
	
	From \eqref{eq:81}, we can conclude that
	$(g+\lambda) v_1 = (\hQ\hB +\hQ W\Phi W\zz + g I_{d-d_n})v_1$.
	Consequently,
	\begin{eqnarray*}
		g+\lambda\in \eig{
			\hQ\hB +\hQ W\Phi W\zz + g I_{d-d_n}
		} = \eig{U\hB U\zz +UW\Phi W\zz U\zz +g I_{d-d_n}},
	\end{eqnarray*}
	which implies
	\begin{eqnarray}\label{eq:87}
	g+\lambda \leq \rho(U\hB U\zz +Y +gI)
	\leq \rho(U\hB U\zz) +\rho(Y+gI)
	=\hl +\rho(Y+gI) \leq \hl + g,
	\end{eqnarray}
	where the last inequality follows from \eqref{eq:newadd1}.
	The relation \eqref{eq:87} directly gives us that
	$\lambda\leq \hl<\frac{4}{3}$. Namely, \eqref{eq:78} also holds in this case.

We have completed the proof.
$\square$
}
\bigskip

\noindent {\bf Appendix B.} {\bf Proof of Lemma \ref{lemma:alg}.}
For convenience, we use the notation
\[
g(\lambda;S,T):= {\rm det}\big[(\lambda-1)^2 I + (2\lambda-1) S + (\lambda-1)
T \big].
\]
We prove this lemma by mathematical induction on the dimension $d$.
When $d=1$, it is easily seen that
\[
g(\lambda;S,T) =  \left\{
\begin{array}{ll}
(\lambda-1)^0 [(\lambda-1)^2  + (2\lambda-1) S + (\lambda-1)T]  & \hbox{if}\;S\neq 0,\\
(\lambda-1)^1 (\lambda-1 +T)  & \hbox{if}\; S=0,\, T\neq 0,\\
(\lambda-1)^2 \cdot1  &\hbox{if}\; S =0,\, T=0,
\end{array}
\right.
\]
which means that Lemma \ref{lemma:alg} holds in this case. Suppose
this lemma is valid for $d\leq k-1$. Consider the case where  $d =k$.

 \begin{itemize}
        \item[Case 1:] $S\succ 0$. In this case, ${\rm Rank}(S) = {\rm
Rank}(S+T)=k$ and then $l = 0$. Because
\[
g(\lambda;S,T)=(\lambda-1)^{l} g(\lambda;S,T) \qquad{\rm and}\qquad
g(1;S,T) = {\rm det}(S) >0,
\]
 Lemma \ref{lemma:alg} holds in this case.

        \item[Case 2:]
         $S\succeq 0$ but not positive definite. Let $S$ admit the
following eigenvalue decomposition
\[
P^\top SP = \left[
\begin{array}{cccccc}
0 &  & & & &\\
&\ddots & &&&\\
& & 0 & &&\\
& & & s_1 & & \\
&&&&\ddots&\\
&&&&& s_t
\end{array}
\right]:=D,
\]
where $P$ is a orthogonal matrix and $s_i>0$. If we let $W = P^\top TP\succeq 0$,
then
\[
g(\lambda;S,T) = g(\lambda;D,W).
\]
The proof proceeds by considering the following two subcases.
\begin{itemize}
\item[Case 2.1:]  $W_{11}=0$. Since $W$ is positive semidefinite, then
$W_{1i}= W_{i1}=0$ for $i=1,2,\ldots,k$. Note that
\[
g(\lambda;D,W) = (\lambda-1)^2 g(\lambda;D',W')
\]
where $D'$ and $W'$ are the submatrices of $D$ and $W$ obtained by deleting the
first row and column. As we have assumed that Lemma \ref{lemma:alg} holds for $d =k-1$
, there exists a polynomial $p(x)$ such that
\[
g(\lambda;D,W) = (\lambda-1)^2(\lambda-1)^{2k-2-{\rm Rank}{D'} -{\rm
Rank}{(D'+W')}}p(\lambda).
\]
Note that ${\rm Rank}(D')={\rm Rank}(D)= {\rm Rank}(S)$ and ${\rm
Rank}(D'+W')={\rm Rank}(D+W)= {\rm Rank}(S+T)$. Thus, we have
\[
g(\lambda;S,T) = (\lambda-1)^{2k- {\rm Rank}(S)-{\rm Rank}(S+T)},
\]
which implies that Lemma \ref{lemma:alg} is true for $d=k$ in this
subcase.
\item[Case 2.2:]  $W_{11}\neq 0$. Without loss of generality, assume
$W_{11}=1$. Let $w^\top =[W_{12},\ldots,W_{1k}]$. By direct
calculation, we obtain
\[
g(\lambda;D,W) = (\lambda-1)^2 g(\lambda;D',W') +
(\lambda-1)g(\lambda;D',W'- ww^\top).
\]
Since ${\rm Rank}(D'+W')\leq {\rm Rank}(D+W) ={\rm Rank}(S+T)$,
there exists a polynomial $p_1(x)$ such that
\[
g(\lambda;D',W')  = (\lambda-1)^{2k-2-{\rm Rank}(S) - {\rm
Rank}(S+T)}p_1(\lambda),
\]
where $p_1(1)\geq 0$. On the other hand, since ${\rm
Rank}(D'+W'-ww^\top) = {\rm Rank}(D+W)-1 ={\rm Rank}(S+T)-1$, there
exists a polynomial $p_2(x)$ such that
\[
g(\lambda;D', W'- ww^\top)= (\lambda-1)^{2k-1-{\rm Rank}(S)-{\rm
Rank}(S+T)}p_2(\lambda),
\]
where $p_2(1)>0$. Therefore,
\[
g(\lambda;S,T) =(\lambda-1)^{2k-{\rm Rank}(S)-{\rm
Rank}(S+T)}(p_1(\lambda)+p_2(\lambda))
\]
and then Lemma \ref{lemma:alg} holds for this subcase.
\end{itemize}
\end{itemize}

This completes the proof.
 \quad $\square$

\bigskip

\noindent {\bf Appendix C.} {\bf Proof of Lemma \ref{lemma:geo}.} It
is easily seen that
\[
{\rm Rank}(S) +{\rm Rank}(\beta A^\top A)= {\rm Rank}
\left[
\begin{array}{cc}
S & 0 \\
0 & \beta AA^\top
\end{array}
\right],
\]
and therefore we need only prove that
\begin{equation}\label{rankeq}
{\rm Rank} \left[
\begin{array}{cc}
S & -A\zz \\
\beta A & 0
\end{array}
\right] = {\rm Rank} \left[
\begin{array}{cc}
S & 0 \\
0 & \beta AA^\top
\end{array}
\right].
\end{equation}
Indeed, consider the following linear system
\begin{equation}\label{system1}
\left[
\begin{array}{cc}
S & -A\zz \\
\beta A & 0
\end{array}
\right] \left[
\begin{array}{c}
x \\
\mu
\end{array}
\right] =0,
\end{equation}
which is equivalent to
\[
\left\{
\begin{array}{l}
Sx - A^\top \mu =0,\\
Ax =0.
\end{array}
\right.
\]
It then holds that
\[
x^\top Sx = x^\top A^\top \mu = (Ax)^\top \mu =0,
\]
and therefore $Sx =0$ and $A^\top \mu=0$, because $S=H+\beta A^\top A$
is positive semidefinite. This means that
\begin{equation}\label{system2}
\left[
\begin{array}{cc}
S & 0 \\
0 & \beta AA^\top \\
\end{array}
\right] \left[
\begin{array}{c}
x \\
\mu
\end{array}
\right] =0.
\end{equation}
On the other hand,  it is not difficult to verify that any solution
of  \eqref{system2} is the solution of \eqref{system1}, in other words, linear systems  \eqref{system1} and
\eqref{system2} are equivalent. As a result, the rank equality
\eqref{rankeq} holds, which completes the proof. \quad $\square$

% ------- reference -------------

\end{document}